\documentclass[british, a4paper,11pt]{article}
\pdfoutput=1
\usepackage[utf8]{inputenc}  

\usepackage[mono=false]{libertine}
\usepackage[T1]{fontenc}  

\usepackage[english]{babel}
\usepackage{tikz}
\usetikzlibrary{positioning,chains,fit,shapes,calc}
\tikzset{cross/.style={cross out, draw=red, minimum size=2*(#1-\pgflinewidth), inner sep=0pt, outer sep=0pt},
cross/.default={3pt}}

\usepackage[a4paper,vmargin={2cm,2cm},hmargin={2.25cm,2.25cm}]{geometry}
\selectfont

\usepackage{amsthm}
\usepackage{amssymb,bbm}
\usepackage[cmintegrals,libertine]{newtxmath}
\usepackage[cal=euler, calscaled=.95, scr=boondoxo]{mathalfa}
\useosf

\usepackage{microtype,enumitem,ulem}
\usepackage{numprint}
\usepackage{framed}
\usepackage{dsfont}
\usepackage{mathbbol}
\usepackage[explicit]{titlesec}

\newtheorem{thm}{Theorem}[section]
\newtheorem{prop}[thm]{Proposition}
\newtheorem{lem}[thm]{Lemma}
\newtheorem{rk}[thm]{Remark}

\newtheorem*{thm*}{Theorem}
\newtheorem*{prop*}{Proposition}
\newtheorem*{lem*}{Lemma}
\usepackage{graphicx}
\usepackage{float}
\restylefloat{figure}

\renewcommand{\P}{\mathbb{P}}

\newcommand{\R}{\mathbb{R}}
\newcommand{\N}{\mathbb{N}}
\newcommand{\Z}{\mathbb{Z}}

\newcommand{\D}{\mathbb{D}}

\newcommand{\cP}{\mathcal{P}}
\newcommand{\cB}{\mathcal{B}}

\newcommand{\cT}{\mathcal{T}}
\newcommand{\cL}{\mathcal{L}}
\newcommand{\cH}{\mathcal{H}}

\newcommand{\cC}{\mathcal{C}}

\newcommand{\be}{\mathbbm{e}}

\newcommand{\Skel}{\mathrm{Sk}}

\newcommand{\cutT}{\mathcal{C}}

\newcommand{\dC}{d_{\mathcal{C}}}
\newcommand{\Sk}{\mathrm{Sk}}

\usepackage{hyperref}
\hypersetup{
	colorlinks=true,
		citecolor=blue!60!black,
		linkcolor=red!60!black,
		urlcolor=green!40!black,
		filecolor=yellow!50!black,
	breaklinks=true,
	pdfpagemode=UseNone,
	bookmarksopen=false,
}

\tikzset{cross/.style={cross out, draw=black, minimum size=2*(#1-\pgflinewidth), inner sep=0pt, outer sep=0pt},
cross/.default={1pt}}
\usepackage{animate}
\usepackage{todonotes}

\DeclareSymbolFont{extraup}{U}{zavm}{m}{n}
\DeclareMathSymbol{\vardspade}{\mathalpha}{extraup}{81}
\DeclareMathSymbol{\varheart}{\mathalpha}{extraup}{86}
\DeclareMathSymbol{\vardiamond}{\mathalpha}{extraup}{87}
\DeclareMathSymbol{\varclub}{\mathalpha}{extraup}{84}

\makeatletter
\renewcommand*{\@fnsymbol}[1]{\ensuremath{\ifcase#1\or  \vardspade \or \varheart \or \vardiamond\or \varclub \or \bigstar \or
   \mathsection\or \mathparagraph\or \|\or **\or \dagger\dagger   \or \ddagger\ddagger \else\@ctrerr\fi}}
\makeatother

\author{
Igor Kortchemski\thanks{CMAP, CNRS, École polytechnique, Institut Polytechnique de Paris, 91120 Palaiseau, France. \textsf{igor.kortchemski@polytechnique.edu}} \qquad \& \qquad
Paul Thévenin\thanks{University of Vienna. \textsf{paul.thevenin@univie.ac.at}}   
}

\title{Coupling Bertoin's and Aldous-Pitman's representations of the  additive coalescent}
\date{}
\begin{document}

\maketitle

\begin{abstract}
We construct a coupling between two seemingly very different constructions of the standard additive coalescent, which describes the evolution of masses merging pairwise at rates proportional to their sums. The first construction, due to Aldous \& Pitman, involves the components obtained by logging the Brownian Continuum Random Tree (CRT) by a Poissonian rain on its skeleton as time increases. The second one, due to Bertoin, involves the excursions above its running infimum of a linear-drifted standard Brownian excursion as its drift decreases. Our main tool is the use of an exploration algorithm of the so-called cut-tree of the Brownian CRT, which is a tree that encodes the genealogy of the fragmentation of the CRT.
\end{abstract}

\section{Introduction}

\paragraph{Cutting down trees.}

Starting with a rooted tree, a natural logging operation consists in choosing and removing one of its edges uniformly at random, thus splitting the tree into two connected components. Iterating and removing edges one after another, one obtains a fragmentation process of this tree. This model was introduced by Meir \& Moon \cite{MM70, MM74} for random Cayley and recursive trees. They focused on the connected component containing the root, and  investigated the number of cuts needed to isolate it. Since then, this subject has brought considerable interest for a number of classical models of deterministic and random trees, including random binary search trees \cite{Hol10,Hol11}, random recursive trees \cite{IM07,DIMRS09,Ber15,BB14} and Bienaymé--Galton--Watson trees conditioned on the total progeny \cite{Ja06,ABH14,AKP06,Pan06,Ber12,BM13}.

It is remarkable that the so-called standard additive coalescent, which describes the evolution of masses merging pairwise at rates proportional to their sums, can be defined (after time-reversal) using a continuous analogue of the cutting procedure  on discrete trees.

\paragraph{The Aldous-Pitman construction.} The Aldous-Pitman fragmentation, introduced in \cite{AP98}, describes the evolution of the masses of the connected components of a Brownian CRT $\cT$ cut according to a Poissonian rain $\cP$ of intensity $\mathrm{d}\lambda \otimes \mathrm{d}t$ on $\Skel(\cT) \times \R_+$, where $\mathrm{d}t$ is the Lebesgue measure on $\R_+$ and $\lambda$ is the length measure on the skeleton $\Skel(\cT)$ on $\cT$ (see Sec.~\ref{ssec:def} for precise definitions). We set, for every $t \geq 0$,
\begin{align*}
\cP_t := \{ c \in \Skel(\cT), \exists \, s \in [0,t] ,(c,s) \in \cP \}.
\end{align*}
Then, for every $t \geq 0$, we define $X_{\mathsf{AP}}(t)$ to be the sequence of $\mu$-masses of the connected components of $\cT \backslash \cP_t$, sorted in nonincreasing order, where $\mu$ is the so-called mass measure on $\cT$. Then $(X_{\mathsf{AP}}(t))_{ t \geq 0}$ is a fragmentation process with explicit characteristics (see \cite{AP98} and more generally the book \cite{Ber06} for a general theory of stochastic coalescence and fragmentation processes, as well as examples and motivation). Up to time-reversal, $X_{\mathsf{AP}}$ is closely related  to the well-known standard additive coalescent \cite{EP98,AP98}.  It is also interesting to mention that very recently  $X_{\mathsf{AP}}$ has naturally appeared in  the study of random uniform factorizations of large permutations into products of transpositions \cite{FK17, The21}.

\paragraph{The Bertoin construction.} Bertoin \cite{Ber00} gave another construction of this fragmentation process from a drifted standard Brownian excursion the following way. Let $\mathbbm{e}$ be a standard Brownian excursion on $[0,1]$. For every $t \geq 0$, consider the function $f_t: [0,1] \rightarrow \R$ defined by $f_{t}(s)=\be_s - t s$ for $s \in [0,1]$ and denote by $X_{\mathsf{B}}(t)$ the sequence of lengths of the excursions of $f_t$ above its running infimum, sorted in nonincreasing order. Bertoin \cite{Ber00} proves that this process has the same distribution as the Aldous-Pitman fragmentation of a Brownian CRT (normalized so that it is coded by $2 \mathbb{e}$), i.e.~that $X_{\mathsf{B}}$ and $X_{\mathsf{AP}}$ have the the same distribution.

It may be puzzling that these two constructions define the same object, since the Aldous-Pitman representation involves two independent levels of randomness (the CRT and the Poissonian rain), while the Bertoin representation only involves a Brownian excursion. For the analog representations in the discrete framework of finite trees, several connections have been recently discovered. We present them just after the statement of Theorem \ref{thm:couplage}, which is our main contribution: in the continuous framework, we unify these two constructions and explain why they are actually intimately related.

\paragraph{Coupling the two constructions.}

Our main result consists in coupling $X_{\mathsf{B}}$ and $X_{\mathsf{AP}}$.
\begin{thm}
\label{thm:couplage}
The following assertions hold.
\begin{itemize}

\item[(i)] Let $\cT$ be a Brownian CRT equipped with the Poissonian rain $ \cP$. On the same probability space, there is a function $F_{(\cT,\cP)}$, measurable with respect to $(\cT,\cP)$, having the law of the Brownian excursion and such that almost surely, for every $t \geq 0$, the nonincreasing rearrangement of the masses of the connected components of $\mathcal{T} \backslash \cP_{t}$ is the same as the nonincreasing rearrangement of the lengths of the excursions of $(F_{(\cT,\cP)}(s)-ts)_{ 0 \leq s \leq 1}$ above its running infimum.

\item[(ii)] Conversely, given a Brownian excursion $\mathbbm{e}$ and an independent sequence of i.i.d.~uniform random variables on $[0,1]$, on the same probability space there is a Brownian CRT $\cT$ equipped with a Poissonian rain $\cP$ (which are measurable with respect to $\mathbbm{e}$ and the latter sequence) such that almost surely $F_{(\cT,\cP)}=\mathbbm{e}$.
\end{itemize}
\end{thm}

Some comments are in order. In (i), the construction of $F_{(\cT,\cP)}$ from ${(\cT,\cP)}$ is explicit using the so-called ``Pac-man algorithm'', which we briefly describe below. In (ii), unlike (i), $\mathbb{e}$ alone is not enough to build $(\cT,\cP)$: an additional independent source of randomness is needed (see Sec.~\ref{sec:recover} for details).

Let us first give some underlying intuition. In the discrete world of finite trees, it turns out that there is a beautiful explicit exact relation, due to Broutin \& Marckert \cite{BM16}, between masses of components obtained after removing edges one after the other, and lengths of excursions above its running infimum of a certain function. The idea is to use the so-called Prim order to explore an edge-labelled tree.  Let us explain this in more detail.

Consider a rooted tree $T_n$  with $n$ vertices, whose edges are labelled from $1$ to $n-1$. Its vertices $u_1, \ldots, u_n$ listed in Prim order are defined as follows: $u_1$ is the root of the tree and, for every $i \in \llbracket 1, n-1 \rrbracket$, $u_{i+1}$ is the vertex, among all children of a vertex of $\{u_j, j \leq i \}$, whose edge to its parent has minimum label.

\begin{figure}[h!]
\begin{minipage}{0.95\linewidth}
\begin{center}
\begin{tikzpicture}
\draw (-1,2) -- (-1.5,1) node [midway, right] {$5$} -- (0,0) node [midway, below] {$9$} -- (0,1) node [midway, right] {$4$} (0,0) -- (1.5,1) node [midway, below] {$11$} (-2,3) -- (-2,2) node [midway, right] {$8$} -- (-1.5,1) node [midway, right] {$12$} -- (-2.5,2) node [midway, left] {$1$} (-.5,2) -- (0,1) node [midway, left] {$3$} -- (.5,2) node [midway, right] {$7$} -- (.5,3) node [midway, right] {$14$} (0,4) -- (.5,3) node [midway, left] {$10$} -- (1,4) node [midway, right] {$13$} (1,2) -- (1.5,1) node [midway, left] {$2$} -- (2,2) node [midway, right] {$6$};
\draw[fill=white] (0,0) circle (.2);
\draw[fill=white] (-1.5,1) circle (.2);
\draw[fill=white] (0,1) circle (.2);
\draw[fill=white] (1.5,1) circle (.2);
\draw[fill=white] (-1,2) circle (.2);
\draw[fill=white] (-2.5,2) circle (.2);
\draw[fill=white] (-2,2) circle (.2);
\draw[fill=white] (-2,3) circle (.2);
\draw[fill=white] (.5,2) circle (.2);
\draw[fill=white] (-.5,2) circle (.2);
\draw[fill=white] (.5,3) circle (.2);
\draw[fill=white] (0,4) circle (.2);
\draw[fill=white] (1,4) circle (.2);
\draw[fill=white] (1,2) circle (.2);
\draw[fill=white] (2,2) circle (.2);
\draw (-.75,.5) node[cross,rotate=60,scale=3] {};
\draw (.75,.5) node[cross,rotate=-60,scale=3] {};
\draw (-1.75,1.5) node[cross,rotate=25,scale=3] {};
\draw (.5,2.5) node[cross,rotate=0,scale=3] {};
\draw (.75,3.5) node[cross,rotate=-25,scale=3] {};
\draw (.25,3.5) node[cross,rotate=25,scale=3] {};
\end{tikzpicture}
\qquad 
\begin{tikzpicture}[every node/.style={scale=0.7}]
\draw (-1,2) -- (-1.5,1) (0,0) -- (0,1) (1.5,1) (-2,3) -- (-2,2) (-1.5,1) -- (-2.5,2) (-.5,2) -- (0,1) -- (.5,2) (1,2) -- (1.5,1) -- (2,2);
\draw[fill=white] (0,0) circle (.2);
\draw[fill=white] (-1.5,1) circle (.2);
\draw[fill=white] (0,1) circle (.2);
\draw[fill=white] (1.5,1) circle (.2);
\draw[fill=white] (-1,2) circle (.2);
\draw[fill=white] (-2.5,2) circle (.2);
\draw[fill=white] (-2,2) circle (.2);
\draw[fill=white] (-2,3) circle (.2);
\draw[fill=white] (.5,2) circle (.2);
\draw[fill=white] (-.5,2) circle (.2);
\draw[fill=white] (.5,3) circle (.2);
\draw[fill=white] (0,4) circle (.2);
\draw[fill=white] (1,4) circle (.2);
\draw[fill=white] (1,2) circle (.2);
\draw[fill=white] (2,2) circle (.2);
\draw (0,0) node{$1$};
\draw (0,1) node{$2$};
\draw (-.5,2) node{$3$};
\draw (.5,2) node{$4$};
\draw (-1.5,1) node{$5$};
\draw (-2.5,2) node{$6$};
\draw (-1,2) node{$7$};
\draw (1.5,1) node{$8$};
\draw (1,2) node{$9$};
\draw (2,2) node{$10$};
\draw (-2,2) node{$11$};
\draw (-2,3) node{$12$};
\draw (.5,3) node{$13$};
\draw (0,4) node{$14$};
\draw (1,4) node{$15$};
\end{tikzpicture}
\qquad 
\raisebox{0.5\height}{
\begin{tikzpicture}[scale=.22]
\draw[->] (0,0) -- (16,0);
\draw[->] (0,-8) -- (0,2);
\draw[dashed,red] (4,-1) --(4,2);
\draw[dashed,red] (7,-2) --(7,2);
\draw[dashed,red] (10,-3) --(10,2);
\draw[dashed,red] (12,-4) --(12,2);
\draw[dashed,red] (13,-5) --(13,2);
\draw[dashed,red] (14,-6) --(14,2);
\draw[dashed,red] (15,-7) --(15,2);
\draw[fill=black] (0,0) circle (.2);
\draw[fill=black] (1,0) circle (.2);
\draw[fill=black] (2,1) circle (.2);
\draw[fill=black] (3,0) circle (.2);
\draw[fill=black] (4,-1) circle (.2);
\draw[fill=black] (5,0) circle (.2);
\draw[fill=black] (6,-1) circle (.2);
\draw[fill=black] (7,-2) circle (.2);
\draw[fill=black] (8,-1) circle (.2);
\draw[fill=black] (9,-2) circle (.2);
\draw[fill=black] (10,-3) circle (.2);
\draw[fill=black] (11,-3) circle (.2);
\draw[fill=black] (12,-4) circle (.2);
\draw[fill=black] (13,-5) circle (.2);
\draw[fill=black] (14,-6) circle (.2);
\draw[fill=black] (15,-7) circle (.2);
\end{tikzpicture}}
\end{center}
\end{minipage}
      \caption{\label{fig:fig1}
      Left: an edge-labelled plane tree. Middle: the forest obtained by labelling the vertices in Prim order and removing the $6$ edges with largest labels. Right: the Prim path of the forest explored using the Prim order.}
\end {figure}
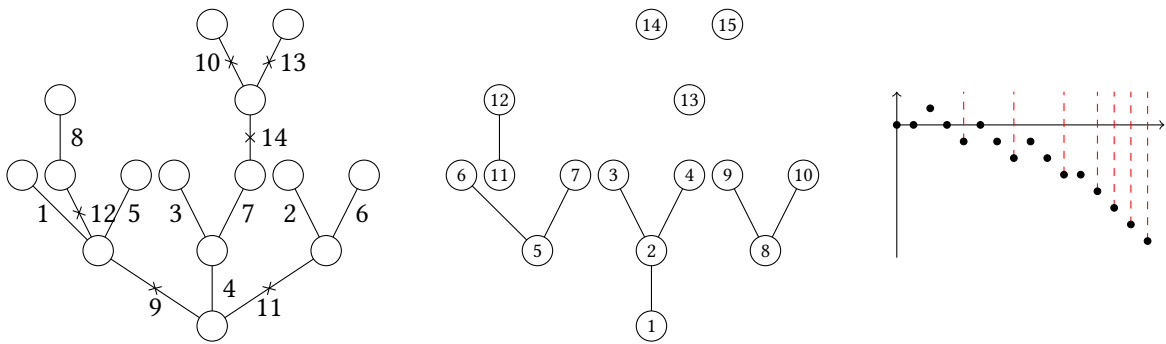

Let $T_n$ be a Cayley tree with $n$ vertices (that is, a tree with $n$ vertices labelled from $1$ to $n$, rooted at $1$), whose edges are labelled from $1$ to $n-1$ uniformly at random conditionally on $T_n$. Let $(u_i)_{1 \leq i \leq n}$ be the vertices of $T_{n}$ sorted according to the Prim order. For every $1 \leq i \leq n$ and $1 \leq k \leq n-1$, let $X_i(k)$ be the number of children of $u_i$ in the forest $F_{n}(k)$ obtained by deleting all  edges with labels belonging to $\llbracket n-k, n-1 \rrbracket$. Finally, set $S_x(k) := \sum_{i=1}^{\lfloor x n \rfloor} (X_i(k)-1)$ for every $0 \leq x \leq 1$ (which is called the Prim path of the forest explored in Prim order). Then Broutin \& Marckert \cite{BM16} establish that:
\begin{enumerate}[noitemsep,nolistsep]
\item[--] the lengths of excursions  of $(S_{x}(k))_{0 \leq x \leq 1}$ above its running infimum are equal to the sizes of the connected components of $F_{n}(k)$ (see Fig.~\ref{fig:fig1}).
\item[--] the following convergence holds in $\D\left( \R_+, \D\left( [0,1], \R \right) \right)$:
\begin{align*}
\left(\frac{\left(S_x(\lfloor t n \rfloor)\right)_{x \in [0,1]}}{\sqrt{n}}\right)_{t \geq 0}  \underset{n \rightarrow \infty}{\overset{(d)}{\rightarrow}} \left((\be_x-t x)_{x \in [0,1]}\right)_{t \geq 0},
\end{align*}
where, for $I \subset \R \cup \{ \pm \infty \}$ an interval and $E$ a metric space, $\D(I,E)$ denotes the space of càdlàg functions from $I$ to $E$ endowed with the $J_{1}$ Skorokhod topology (we refer to Annex $A2$ in \cite{Kal02} for further definitions and details).
\end{enumerate}
Using the fact that the Aldous-Pitman fragmentation is the continuum analog of this discrete logging procedure, this gives another proof of the fact that $X_{\mathsf{B}}$ and $X_{\mathsf{AP}}$ have the same distribution. This also indicates that if one couples $X_{\mathsf{B}}$ and $X_{\mathsf{AP}}$, then the Brownian excursion appearing in the definition of $X_{\mathsf{B}}$ should represent, in a certain sense, the encoding of the exploration of a Brownian CRT equipped with a Poissonian rain using an associated Prim order.

Also, still in the discrete world of finite trees, Marckert \& Wang \cite{MW19} couple a uniform Cayley tree with $n$ vertices and edges labelled from $1$ to $n-1$ uniformly at random with a uniform Cayley tree equipped with an independent uniform decreasing edge-labelling (i.e.~labels decrease along paths directed from the root towards the leaves) in such a way that edge removals give the same sizes of connected components. A connection with discrete cut-trees is also given (see \cite[Sec.~3]{MW19}). Indeed, it turns out that this second tree is closely related to the cut-tree of the first one, which allows Marckert \& Wang to give a nice simple proof of the well-known fact that the number of cuts needed to isolate a uniform vertex in a uniform Cayley tree, scaled by $\sqrt{n}$, converges in distribution to a Rayleigh random variable (see  also \cite{MW19} for other connections between the standard additive coalescent and other combinatorial and probabilistic models such as size biased percolation and parking schemes in a tree).

However, how to make such statements precise in the realm of continuous trees remains unclear. For this reason, we use a different route to define a coupling between  $X_{\mathsf{B}}$ and $X_{\mathsf{AP}}$.
The main tool in the proof of Theorem \ref{thm:couplage} is the use of the so-called cut-tree $\cutT$, defined by Bertoin and Miermont in \cite{BM13}, which roughly speaking encodes the genealogy of the fragmentation of a Brownian CRT by a Poissonian rain. Indeed, one of our contributions is to use the cut-tree to define the ``Bertoin'' excursion $F_{(\cT,\cP)}$ from the Aldous-Pitman fragmentation by using an algorithm which we call the ``Pac-Man'' algorithm,  roughly described as follows (see Sec.~\ref{ssec:pacman} for a precise definition). With every value $h \in [0,1]$, using a local exploration procedure, we associate a final target point of $\cutT$. Then the value $F_{(\cT,\cP)}(h)$ is defined using the genealogy of this exploration. Strictly speaking, the cut-tree $\cutT$ is defined from $(\cT,\cP)$ by using additional randomness (namely a collection of independent i.i.d.~points sampled according to the mass measure). However, the quantity $F_{(\cT,\cP)}$ can be directly defined as a measurable function of $(\cT,\cP)$ without any reference to the cut-tree, which still serves as a useful tool to check that  $F_{(\cT,\cP)}$ satisfies the desired properties (see Sec.~\ref{ssec:proof} for details).

Conversely, given the ``Bertoin'' excursion, we will also see that coupling $X_{\mathsf{B}}$ and $X_{\mathsf{AP}}$ is closely related to the question of reconstruction of the original CRT from its cut-tree (see Sec.~\ref{sec:recover} for more details). 

Quite interestingly it seems that, while there is no simple analog of the coupling between drifted excursions and sizes of connected components using Prim's order in the continuous framework, there is no simple analog of the coupling between drifted excursions and sizes of connected components using the cut-tree in the  discrete framework. Indeed, in Marckert \& Wang's coupling, given the Cayley tree, the decreasing edge labeling is random, while in the continuous framework, given the cut-tree, its labelling is deterministic; see Remark \ref{rem:labels}. {In particular, new ideas and techniques are required to analyze the continuous framework.}

Finally, let us mention that the question of reconstructing the Brownian CRT from the ``Bertoin'' excursion appears in an independent work by Nicolas Broutin and Jean-Fran\c cois Marckert \cite{BM23} in the different context of the study of scaling limits of minimal spanning trees on complete graphs.

\paragraph{Perspectives.}  There has been some recent developments in studying analogs of the Aldous-Pitman fragmentation for different classes of trees and their associated cut-trees, see \cite{Die15,BW17b,BH20,BHW22,ADG19}. In particular, it has been shown that by an appropriate tuning (fragmentation alongs the skeleton and/or at nodes), the law of the cut-tree is equal to the law of the original tree.

We expect that our coupling can be extended to more general classes of trees, with fragmentation on the skeleton only, such as stable trees (the study of this fragmentation is mentioned in \cite[Section $5$, (6)]{ADG19}). Indeed, for stable trees, it is known \cite{BH20} that the analog of Aldous-Pitman fragmentation can be obtained using the normalized excursion of a stable process. One of the main issues is that in this case the associated cut-tree is not compact anymore; hence, our main argument, which consists in comparing distances in this cut-tree, does not apply directly. Furthermore, new results (see \cite{Wan22}) suggest strong connections between the so-called ICRT (Inhomogeneous Continuum Random Trees) and stable trees. Thus similar techniques could work in both cases, and it is plausible that analog couplings exist. We plan to investigate this in future work.

\paragraph{Acknowledgments.} I.K. thanks Jean Bertoin for asking him if there is a connection between the two seemingly different constructions $X_{\mathsf{B}}$ and $X_{\mathsf{AP}}$, and we thank Jean-Fran\c cois Marckert for mentioning the use of cut-trees in \cite{MW19}. We also  thank the  anonymous referees for their  careful review of the paper and for their constructive comments which have greatly contributed to improve the paper.

\paragraph{Overview of the paper.} Section \ref{sec:cuttree} is devoted to the definition of our main tool, the cut-tree associated to the fragmentation of the Brownian CRT; we also prove there some preliminary structural results on this object. In Section \ref{sec:coupling}, we prove the first part of our main result, Theorem \ref{thm:couplage}, with the help of our so-called Pac-Man algorithm. Finally, we prove Theorem \ref{thm:couplage} (ii) in Section \ref{sec:recover}, essentially making use of results from \cite{ADG19}.

\tableofcontents

\section{The cut-tree of the Brownian CRT}
\label{sec:cuttree}

An important object in our study is the cut-tree of a Brownian CRT, which roughly speaking encodes the genealogy of its fragmentation by a Poissonian rain. We recall here its construction and main properties, and refer to \cite{BM13,ADG19} for details and proofs.

\subsection{Definitions}
\label{ssec:def} 

Let us first introduce some definitions and notation for trees.

\paragraph{Real trees.} We say that a complete metric space $(T,d)$ is a real tree if, for every $u,v \in T$:
\begin{itemize}
\item there exists a unique isometry $f_{u,v}: [0,d(u,v)] \rightarrow T$ such that $f_{u,v}(0)=u$ and $f_{u,v}(d(u,v))=v$.
\item for any continuous injective map $f:[0,1] \rightarrow T$ such that $f(0)=u$ and $f(1)=v$, we have $f([0,1]) = f_{u,v}([0,d(u,v)]) =: \llbracket u, v \rrbracket$.
\end{itemize}
A rooted real tree is a real tree with a distinguished vertex, called the root.

\paragraph{Tree structure.} Let $\cT$ be a real tree. We say that a point $x \in \cT$ is a leaf if $\cT \backslash \{ x \}$ is connected, and a branchpoint if $\cT \backslash \{ x \}$ has at least three connected components.   We  denote by $\cB(\cT)$ the set of all branchpoints of the tree $\cT$ and  by $\varnothing$ the root of $\cT$. 

We let $\Sk(\cT)$ be the skeleton of $\cT$, that is, the set of all points of $\cT$ that are not leaves nor branchpoints. We emphasize that this definition differs from the usual one, where the skeleton is defined as the complement of the set of leaves. The reason is that in the sequel of the paper, we usually need to treat differently leaves, branchpoints, and non-branchpoints which are not leaves.

We denote by $\cT_{x}$ the tree which is the set of all (weak) descendents of $x$ in $\cT$, rooted at $x$. If the tree $\cT$ is binary (that is, for all $x \in \cT$, $\cT \backslash \{ x \}$ has at most three connected components, and $\cT \backslash \{ \varnothing \}$ has at most two), when $x$ is an ancestor of $y$, we denote by $\cT^{y}_{x}$ the subtree above $x$ containing $y$, rooted at $x$ and by $\bar{\cT}^{y}_{x}$ the subtree above $x$ \emph{not} containing $y$, rooted at $x$ (which is unique if it exists). 

Furthermore, for any two vertices $x,y \in \cT$, we write $x \preceq y$ when $x$ is an ancestor of $y$, and $x \prec y$ when $x \preceq y$ and $x \neq y$. In particular $\varnothing \preceq x$ for every $x \in \cT$ and $\preceq$ is a partial order on $\cT$ called the genealogical order.

Finally, for $x, y \in \cT$, we denote by $x \wedge y$ their closest common ancestor, that is, the unique $z \in \cT$ satisfying $\llbracket \varnothing, x \rrbracket \cap \llbracket \varnothing, y \rrbracket = \llbracket \varnothing, z \rrbracket$.

\paragraph{Brownian excursion and Brownian tree.}

The Brownian CRT, introduced by Aldous \cite{ald1,ald2,ald3}, is a random real tree constructed from twice a standard Brownian excursion $\be: [0,1] \rightarrow \R_+$ the following way. The function $2 \be$ induces an equivalence relation $\sim_{2 \be}$ on $[0,1]$: define a pseudo-distance $d$ on $[0,1]$ by setting $d(u,v) = 2 \be_u + 2 \be_v - 2 \min_{[u,v]} 2 \be$, and say that, for all $0 \leq u,v \leq 1$, $u \sim_{2\be} v$ if and only if $d(u,v)=0$. Define now $\cT \coloneqq [0,1] / \sim_{2\be}$, endowed with the distance which is the projection of $d$ on the quotient space (which we also denote by $d$ by convenience). It is standard that $(\cT,d)$ is a real tree, called the Brownian CRT.

\paragraph{Length and mass measures.}

For any real tree $(T,d)$, observe that the distance $d$ on $T$ induces a length measure $\lambda$ on $\Sk(T)$, defined as the only $\sigma$-finite measure such that $\lambda(\llbracket u, v \rrbracket) = d(u,v)$ for all $u,v \in T$. 
In the case of the Brownian CRT $(\cT,d)$, we can furthermore endow it with a mass measure $\mu$, which is the pushforward of the Lebesgue measure on $[0,1]$ by the quotient map $[0,1] \rightarrow [0,1] / \sim_{2 \be}$. Roughly speaking, $\mu$ accounts for the proportion of leaves in a given component of $\cT$.

\subsection{Construction of the Brownian cut-tree}

\label{ssec:cuttree}

We first need some notation. Let $\cT$ be a  Brownian CRT, $\mu$ its mass measure and let $\cP$ be a Poissonian rain of intensity $\mathrm{d}\lambda \otimes \mathrm{d}t$, where $\lambda$ denotes the length measure on $\Sk(\cT)$ and $\mathrm{d}t$ is the Lebesgue measure on $\R_+$. For every $t \geq 0$, we set
\begin{align*}
\cP_t := \{ c \in \Skel(\cT), \exists \, s \in [0,t], (c,s) \in \cP \}.
\end{align*}
The elements of $\cP_\infty \coloneqq \cup_{t \geq 0} \cP_t$ are called \emph{cutpoints}. For every $t \geq 0$ and $ x \in \cT$, we denote by $\cT_{t}(x)$ the connected component of $ \cT \backslash \cP_{t}$ containing $x$ and $\mu_{t}(x)=\mu(\cT_{t}(x))$ its $\mu$-mass. If $x \in \cP_t$, we set $\cT_t(x)=\varnothing$ and $\mu_t(x)=0$.

Let $U_{0}=\varnothing$ be the root of $\cT$, and let $(U_i)_{i \geq 1}$ be a sequence of i.i.d.~leaves of $\cT$ sampled according to the mass measure $\mu$, independently of $\cP$. For every $i,j \in \Z_{+}$, we let $t_{i,j} \coloneqq \inf\{ t \geq 0, \cT_t(U_i) \neq \cT_t(U_j) \}$ be the first time a cutpoint appears on $\llbracket U_{i},U_{j} \rrbracket$. Then by \cite{BM13, ADG19} there exists almost surely a metric space $C^\circ := (C^\circ, d^\circ, \rho)$ containing the set $\{ \rho \} \cup \Z_{+}$ such that $C^\circ = \bigcup_{i \in \Z_{+}} \llbracket \rho, i \rrbracket $ and for every $i,j \in \Z_{+}$,
\[  d^\circ(\rho, i)=\int_0^\infty \mu_{s}(U_{i}) \, \mathrm{d}s \qquad \mathrm{and} \qquad   d^\circ(i,j)=\int_{t_{i,j}}^\infty \left(\mu_s(U_i) +  \mu_s(U_j) \right)\mathrm{d}s.
\]
We denote by $\cutT \coloneqq (\cutT ,d_{\cutT},\rho)$ the completion of this metric space, which is a real tree called the cut-tree. The sequence $(i)_{i \geq 1}$ is dense in $ \mathcal{C}$, and  in particular every branchpoint of $ \mathcal{C}$ can be written (non uniquely) as $i \wedge j$ with $i,j \geq 1$. We also endow the set of leaves of $\cutT$ with the measure $\nu$ defined as:
\begin{equation}
\label{eq:defnu}
\nu= \lim_{n \rightarrow \infty} \frac{1}{n} \sum_{i=1}^n \delta_i.
\end{equation}
For a measurable subset $\mathcal{A} \subset \cC$ with $\nu(\mathcal{A})>0$ we use the notation $\nu_{\mathcal{A}}$ for the probability measure on $\cC$ defined by $\nu_{\mathcal{A}}(B)=\nu(\mathcal{A} \cap B)/\nu(\mathcal{A})$.

The main result of \cite{BM13} is the following:

\begin{thm}[Bertoin \& Miermont \cite{BM13}]
\label{thm:cuttreecrtiscrt}
The cut-tree of the Brownian CRT has the law of a Brownian CRT:
\begin{align*}
\mathcal{C} \overset{(d)}{=} \mathcal{T}.
\end{align*}
\end{thm}

The cut-tree $\cutT$ encodes the genealogical structure, as time increases,  of the cuts of the subtrees which contain the points $(U_{i})_{i \geq 0}$, in such a way that branchpoints of $\cutT$ are in correspondence with $ \mathcal{P}_{\infty}$. Let us  make this  more explicit.

Informally speaking, for every $i,j \in \Z_{+}$, the branchpoint $i \wedge j$ of $\cutT$ encodes the (a.s. unique) cutpoint appearing on $\llbracket U_{i}, U_{j} \rrbracket$ at  time  $t_{i,j}$. The subtree  of $\cutT$ above $i \wedge j$ containing $i$  (resp. $j$) is then the cut-tree of the subtree of $ \mathcal{T} \backslash \mathcal{P}_{t_{i,j}}$ containing $U_{i}$ (resp.~$U_{j}$). 

The measures $\mu$ on $\mathcal{T}$ and $\nu$ on $\cutT$ are related in the following way (see \cite[Proposition $7$]{ADG19}):
\begin{equation}
\label{eq:lienmunu} \mu ( \mathcal{T}_{t_{i,j}}(U_{i}))= \nu( \cutT^{i}_{i \wedge j}), \qquad  \mu ( \mathcal{T}_{t_{i,j}}(U_{j}))= \nu( \cutT^{j}_{i \wedge j}).
\end{equation}
In particular, using the fact that $\cutT$ is binary, since $\mathcal{T}_{t_{i,j}-}(U_{i})= \mathcal{T}_{t_{i,j}-}(U_{j})= \mathcal{T}_{t_{i,j}}(U_{i}) \cup  \mathcal{T}_{t_{i,j}}(U_{j}) \cup \{ i \wedge j \}$,  we have $\mu(\mathcal{T}_{t_{i,j}-}(U_{i}))=\nu( \mathcal{C}_{i \wedge j})$. 

The leaf $0 \in \mathcal{C}$ will play a distinguished role in the sequel. In the terminology of \cite[Section 3.3]{ADG19}, it can be seen as the ``image'' in $\mathcal{C}$ of the root $\varnothing$ of $\cT$ in the following sense: if $(U_{i_{n}})_{n \geq 1}$ is a sequence converging to $\varnothing$, then $i_{n}$ converges to $0$ (this is shown in  \cite[Section 3.3]{ADG19}). Similarly, for every branchpoint $x$ of $\cutT$, each of the two subtrees grafted on $x$ comes with a distinguished leaf. Indeed, consider
$x \prec y$ two points of $\mathcal{C}$ with $x$ a branchpoint. Recall that $\mathcal{C}^{y}_{x}$ denotes the subtree above $x$ containing $y$, rooted at $x$ and $\bar{\mathcal{C}}^{y}_{x}$ the subtree above $x$ \emph{not} containing $y$, rooted at $x$. Intuitively speaking,  $\Lambda(\cC^{y}_{x})$  is the image of $x$ obtained by considering the Poissonian rain in the the subtree $\mathcal{C}^{y}_{x}$ rooted at $x$, and $\Lambda(\bar{\cC}^{y}_{x})$  is the image of $x$ obtained by considering the Poissonian rain in the the subtree $\bar{\mathcal{C}}^{y}_{x}$ rooted at $x$.

More formally, we may find $i,j \geq 1$ such that  $x=i \wedge j$, $i \in \mathcal{C}^{y}_{x}$ and  $j \in   \bar{\mathcal{C}}^{y}_{x}$.
Let $c \in \llbracket U_{i},U_{j} \rrbracket$ be the cutpoint appearing at time $t_{i,j}$.
Consider a sequence $(U_{i_{n}})_{n \geq 1}$ of elements of $\cT_{t_{i,j}}(U_{i})$ converging to $c$ and a sequence $(U_{j_{n}})_{n \geq 1}$ of elements of $\cT_{t_{i,j}}(U_{j})$ converging to $c$. Then by \cite[Section 3.3]{ADG19}
the sequence $(i_{n})_{n \geq 1}$ converges in $\cutT^{y}_{x}$ to a leaf denoted by $\Lambda(\cC^{y}_{x})$  (which does not depend on the sequence $(U_{i_{n}})_{n \geq 1}$), and the sequence $(j_{n})_{n \geq 1}$ converges in $\bar{\cutT}^{y}_{x}$ to a leaf denoted by $\Lambda(\bar{\cC}^{y}_{x})$ (which does not depend on the sequence $(U_{j_{n}})_{n \geq 1}$). To see that the limiting points have to be leaves, simply observe that $t_{i_n , i_{n+1}} \underset{n \rightarrow \infty}{\rightarrow} + \infty$, so that 
\begin{align*}
\nu \left( \cutT_{i_n \wedge i_{n+1}}^{i_n} \right) = \mu_{ t_{i_n , i_{n+1}}} \left( U_{i_n}\right) \underset{n \rightarrow \infty}{\rightarrow} 0.
\end{align*}

Finally, setting  for every $x \in \cutT$:
\begin{equation}
\label{eq:deft}\tau_{x}= \int_{\llbracket \rho ,x \rrbracket}  \frac{1}{\nu(\cutT_{z} )} \lambda(\mathrm{d} z),
\end{equation}
we have $t_{i,j}=\tau_{i \wedge j}$ (see e.~g.~the end of the proof of Theorem 16 in \cite{ADG19}). In other words, the times at which cutpoints appear can be recovered from the cut-tree. Observe that $\tau$ is increasing along any branch of $\cutT$.

We end this section with a result which tells how to find the connected components of $\cT \backslash \cP_t$ using the cut-tree.

\begin{lem}
\label{lem:bijectionsubtreescomponents}
Fix $t > 0$. The connected components of $\cT \backslash \cP_t$ are in bijection with the subtrees of $\cutT$ of the form $\cutT_x^{y}$, with $x,y \in \mathcal{C}$ satisfying $\tau_x=t$ and $x \prec y$. Furthermore, this bijection conserves the masses.
\end{lem}

\begin{proof}
Let $C$ be a connected component of $\cT \backslash \cP_t$, and let $x_C \in \cutT$ be the most recent common ancestor of the leaves $\cL_C \coloneqq \{ i \in \Z_{+}: U_i \in C \}$.
Notice that $x_C$ is not necessarily  a branchpoint (indeed $x_{C}$ belongs to the skeleton when the component $C$ is the same at times $t$ and $t-$).  

 First let us show that $\tau_{x_C}=t$. Since $\cutT$ is a CRT, almost surely there exists a sequence of branchpoints of the form $i_{n} \wedge j_{n}$ converging to $x$ with $i_{n}, j_{n} \in \cL_C $. For all $n$, since $U_{i_{n}}$ and $U_{j_{n}}$ are in the same connected component of  $\cT \backslash \cP_t$ it follows that $\tau_{i_{n} \wedge j_{n}} \geq t$ for every $n \geq 1$. Since $z \mapsto \tau_z$ is continuous, we obtain that $\tau_{x_C} \geq t$. Now assume by contradiction that $\tau_{x_C}>t$. Then, again by continuity of $z \mapsto \tau_z$ and since $\cutT$ is a CRT, there exists a branchpoint $b$ such that $b \prec x_C$ and $\tau_b > t$. Now take two leaves $j \in \overline{\cutT}_b^{x_C}$ and $i \in \cL_C$. We have $i \wedge j = b$ and thus $t_{i,j}=\tau_b>t$, so that $j \in \cL_C$. This contradicts the definition of $x_C$. Hence, $\tau_{x_C} = t$.

Now take $i \in \cL_C $ and set
$$\Phi(C)= \cutT_{x_{C}}^{i}.$$
Let us check that $\Phi$ is well defined by showing that for $i,j \in \cL_C$ we have $\cutT_{x_{C}}^{i}=\cutT_{x_{C}}^{j}$. To this end, observe that by definition $x_C \preceq i \wedge j$, and argue by contradiction assuming that $i \wedge j=x_{C}$. Then $t_{i,j}=\tau_{x_C}=t$, so that $U_{i}$ and $U_{j}$ do not belong to the same connected component of $\cT \backslash \cP_t$. Hence, $i$ and $j$ cannot be both in $\cL_C$, which leads to a contradiction.

Finally, let us establish that $\Phi$ is bijective. To this end, we exhibit the reverse bijection. Consider a subtree of $\cutT$ of the form $\cutT_x^{y}$, with $x \in \mathcal{C}$ satisfying $\tau_x=t$ and $x \prec y$. Consider $i \in \Z_{+}$ such that $i \in \cutT_x^{y}$ and denote by $C$ the connected component of $\cT \backslash \cP_t$ containing $U_{i}$. We set
$$ \Psi(\cutT_x^{y})= C.$$
The map $\Psi$ is well defined since, if  $i,j \in \cutT_x^{y}$ then $t_{i,j}>\tau_x = t$ so that the connected component of $\cT \backslash \cP_t$ containing $U_{i}$ also contains $U_{j}$.

Now, if  $C$ is a connected component of $\cT \backslash \cP_t$, then $\Psi \circ \Phi (C)=C$. Indeed, let $i \in \Z_+$ be such that $U_{i} \in C$. By definition of $\Phi$, $i \in \Phi(C)$. In turn by definition of $\Psi$, $\Psi \circ \Phi(C)$ is the connected component of $\cT \backslash \cP_t$ containing $U_{i}$, which is precisely $C$.

Conversely, consider a subtree of $\cutT$ of the form $\cutT_x^{y}$, with $x \in \mathcal{C}$ satisfying $\tau_x=t$ and $x \prec y$. We check that $\Phi \circ \Psi(\cutT_x^{y})=\cutT_x^{y}$. Let $i \in \Z_+$ be such that $i \in \cutT_x^{y}$. Then $C=\Psi(\cutT_x^{y})$ is the connected component of $\cT \backslash \cP_t$ containing $U_{i}$. It follows that $\Phi(C)= \cutT_{x_{C}}^{i}$. In particular, $x,x_{C} \in \llbracket \rho, i \rrbracket$ and $\tau_{x}=\tau_{x_{C}}=t$. Since $\tau$ is increasing on $ \llbracket \varnothing, i \rrbracket$, it follows that $x=x_{C}$, and thus $ \cutT_x^{y}=\cutT_{x_{C}}^{i}$. This completes the proof.

The fact that this bijection conserves the masses is a consequence of the fact that $\Phi(C)= \overline{\cup_{i : U_{i} \in C} \llbracket x_{C},i \rrbracket}$ for every connected component $C$  of $\cT \backslash \cP_t$ and that $\nu= \lim_{n \rightarrow \infty} \frac{1}{n}\sum_{i=1}^n \delta_i$ and $\mu=\lim_{n \rightarrow \infty}\frac{1}{n}\sum_{i=1}^n \delta_{U_i}$.
\end{proof}

Recall that our main result, Theorem \ref{thm:couplage}, consists in coupling  both processes $X_{\mathsf{AP}}$ and $X_{\mathsf{B}}$. First in Sec.~\ref{sec:coupling} we start from the Aldous-Pitman fragmentation on a CRT and we construct an excursion-type function, which we show to be continuous and to be equal in law to a Brownian excursion, thus proving Theorem \ref{thm:couplage} (i). Then in Sec.~\ref{sec:recover} we explain how to recover the Brownian CRT with its Poissonian rain from the ``Bertoin excursion'', proving Theorem \ref{thm:couplage} (ii).

\section{Defining the ``Bertoin'' excursion from the Aldous-Pitman fragmentation}

\label{sec:coupling}

Here we start from a Brownian CRT $\cT$ equipped with a Poissonian rain $ \cP$, and using the the cut-tree $\cutT$ of $\cT$ we construct a function $F$ having the law of a Brownian excursion, in such a way that, for all $t \geq 0$, the nonincreasing rearrangement of the masses of the connected components of $\mathcal{T} \backslash \cP_{t}$ are the same as the nonincreasing rearrangement of the lengths of the excursions of $(F(s)-ts)_{ 0 \leq s \leq 1}$ above its running infimum.

The algorithm used to construct the function $F$, which we call the Pac-Man algorithm and which we define in Sec.~\ref{ssec:pacman}, consists in exploring the cut-tree $\cutT$ from its root $\rho$, associating with each element $h \in [0,1]$ a final target point $\pi_h$ in $\cutT$ in a surjective way, and a value $F(h)$. We then investigate the properties of this function $F$, showing that it is continuous (Sec.~\ref{ssec:continuity}) and that it has the law of a Brownian excursion (Sec.~\ref{ssec:exc}).

\subsection{Defining an excursion-type function from the Aldous-Pitman representation}
\label{ssec:pacman}

We keep the notation of Section \ref{ssec:cuttree}. Here we shall construct an excursion-type function $F$ from a Brownian CRT $\cT$ equipped with a Poissonian rain $ \cP$ which will turn out to meet the requirements of Theorem \ref{thm:couplage} (i). To this end, we shall use the Brownian cut-tree $\cutT$ associated with $\cT$ as defined in Section~\ref{sec:cuttree}.

With every value $h \in [0,1]$ we shall associate one point $\pi_h$ of $\cutT$ using a recursive procedure.
 It can be informally presented as follows. Imagine Pac-Man starting at the root of $\cutT$ and wanting to eat exactly an amount $h$ of mass. It has an initial target leaf $\ell$, and starts going towards this target. As soon as it encounters a point $x$ such that the subtree $\cutT_x^\ell$ containing the target leaf has mass at most $h$, Pac-Man eats this subtree; if this mass was strictly less than $h$, then it turns out that this point was necessarily a branchpoint, and Pac-Man continues his journey in the remaining subtree equiped with a new target. Pac-Man stops when it has eaten an amount $h$ of mass, and we denote by $\pi_h$ the ending point of the process.

Given a tree $T=(T, r, \nu)$ with root $r$ and mass measure $\nu$,  a distinguished leaf $\ell$ and a value $0 \leq h \leq  \nu(T)$, we set
$$B \left( T, \ell,h \right)= \inf \left\{x \in \llbracket r, \ell \rrbracket : \nu(T_{x}^{\ell}) \leq h \right\}.$$
Observe that $B ( T, \ell,0 )=\ell$, $B ( T, \ell, \nu(T))= r$ and that for $0 < h <  \nu(T)$,  if $B ( T, \ell,h )$ is a point of the skeleton of $T$, then necessarily $\nu(T_{B(T, \ell, h)}^\ell)=h$.

\bigskip

\noindent\textbf{Pac-Man algorithm.} Given $h \in [0,1]$, we define a sequence $(B_{i},L_{i},H_{i})_{0 \leq i < N+1}$ with $N \in \mathbb{Z}_{+} \cup \{+\infty\} $  as follows (we drop the dependence in $h$ to simplify notation). First, set $B_{0}= \rho$, $L_{0}= 0$, $H_{0}=h$. Then, by induction, if  $(B_{i},L_{i},H_{i})_{0 \leq i \leq k}$ have been constructed, we set:
\[B_{k+1}=B \left( \cutT_{B_{k}}^{L_{k}},  L_{k},H_{k} \right), \qquad H_{k+1}=h-\nu \left(\cutT_{B_{k+1}}^{L_{k}}\right).\]
If $H_{k+1}=0$, we set  $N=k+1$ and stop, otherwise we set $L_{k+1}= \Lambda\left(\bar{\cC}_{B_{k+1}}^{L_{k}}\right)$ and continue (see Fig.~\ref{fig:pacman} for an illustration).  In particular, observe that $h=\sum_{1 \leq k <N+1}\nu (\cutT_{B_{k}}^{L_{k-1}})$ by construction.
When $N<\infty$, we set $\pi_h=B_{N}$. When $N=\infty$, we define $\pi_h$ as the point which is the limit  $(B_{i})_{i \geq 0}$ (we will later see that when $N=\infty$ the point $\pi_h$ is necessarily a leaf). Observe that the limit exists by compactness, since the sequence  $(B_{i})_{i \geq 0}$ is increasing for the genealogical order.

  Finally, we  set
 \begin{equation}
 \label{eq:defF}F(h)=\sum_{1 \leq k <N+1} \tau_{B_{k}} \cdot \nu \left(\cutT_{B_{k}}^{L_{k-1}}\right),
 \end{equation}
 where we recall from \eqref{eq:deft} the notation $\tau_{x}$ for $x \in \cutT$. In order to unify the treatment, here and in the sequel, when $N=\infty$, the notation $\sum_{1 \leq i < N+1}$ simply means $\sum_{i=1}^{\infty}$; similarly $(\cdot)_{1 \leq i <N+1}$ means $(\cdot)_{i \geq 1}$.
  
 \begin{figure}[!ht] \centering
\includegraphics[scale=0.6]{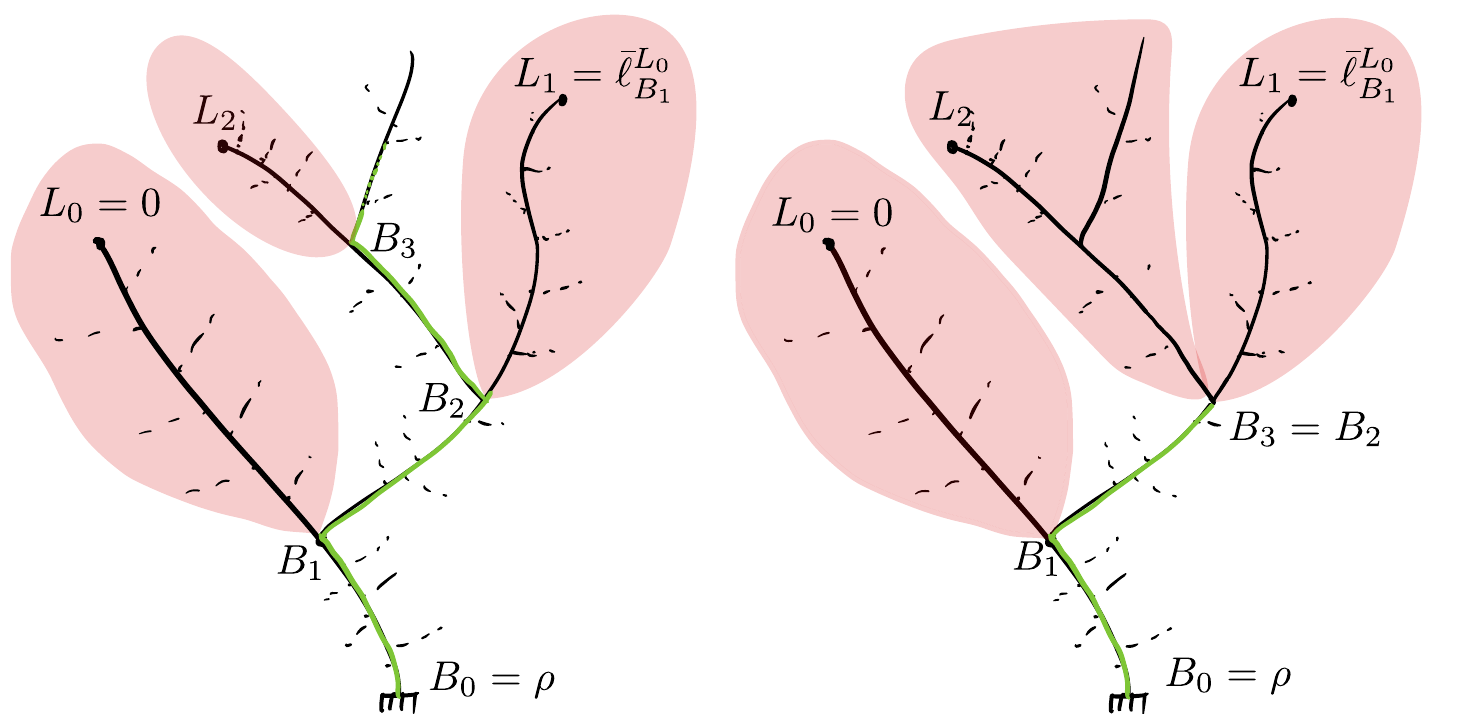}
\caption{Two illustrations of the construction: Pac-Man's trajectory in the cut-tree is represented in green and the eaten subtrees in red. For example, when reaching $B_{1}$, Pac-Man eats the subtree containing $L_{0}$ (it is the first time the mass of the subtree above is less than $h$), and continues in direction of $L_{1}$. The value $F(h)$ is obtained by summing weighted masses of the eaten subtrees.}
\label{fig:pacman}
\end{figure}

\begin{rk}
Observe that the Pac-Man algorithm is initialized with the target leaf $0$, which is the ``image'' of $\varnothing$ in the cut-tree $\mathcal{C}$. Changing the initial target leaf amounts to rerooting the tree $\cT$, and will produce another function $F$ (for which the assertion of Theorem \ref{thm:couplage} (i) still holds).
\end{rk}

It is rather straightforward to check that \eqref{eq:defF}  defines a bounded function on $[0,1]$.
  
\begin{lem}
\label{lem:bounded}
We have $
\sup_{h \in [0,1]} F(h) \leq \mathsf{Height}(\cutT) $.
\end{lem}

\begin{proof}
Take $h \in [0,1]$ and keep the previous notation.
To simplify notation, set $m_{k}= \nu (\cutT_{B_{k}} ^{L_{k-1}})$ for $1 \leq k<N+1$. Since $B_{i}$ is an ancestor of $B_{j}$ for $i<j$, for every $0 \leq i <N+1$ we have
$$\int_{\llbracket B_{i},B_{i+1} \rrbracket }  \frac{1}{\nu(\cutT_{z})} \lambda(\mathrm{d} z) \leq  \frac{d(B_{i},B_{i+1})}{ \sum_{i+1 \leq j <N+1} m_{j}}.$$
It readily follows that for every $1 \leq k <N+1$:
$$\tau_{B_k} = \int_{\llbracket \rho,  B_{k} \rrbracket }  \frac{1}{\nu(\cutT_{z} )} \lambda(\mathrm{d} z) \leq \sum _{i=0}^{k-1} \frac{d(B_{i},B_{i+1})}{ \sum_{i+1 \leq j <N+1} m_{j}}.$$
Thus
$$F(h) \leq  \sum_{1 \leq k<N+1} m_k \sum _{i=0}^{k-1} \frac{d(B_{i},B_{i+1})}{ \sum_{i+1 \leq j <N+1}m_{j}}= \sum_{0 \leq i< N+1 } \sum_{i+1 \leq k<N+1} \frac{m_{k}}{\sum_{i+1 \leq j <N+1} m_{j}} d(B_{i},B_{i+1})=\sum_{0 \leq i<N+1} d(B_{i},B_{i+1}),$$
and the desired conclusion follows.
\end{proof}

\begin{rk}
\label{rem:labels}
Given the rescaled convergence of discrete cut-trees to continuous cut-trees \cite{BM13}, it is natural to expect that when one equips the discrete cut-trees with labels corresponding to cutting times, a joint convergence holds towards $\cutT$ equipped with the labelling $(\tau_{x})_{x \in \cutT}$. It is interesting to notice that in the discrete case, given the cut-tree, the labelling is random (see \cite[Sec.~3]{MW19}), while in the continuous case, given $ \cutT$, the labelling $(\tau_{x})_{x \in \cutT}$ is deterministic.
\end{rk}

\subsection{Continuity of the function $F$}
\label{ssec:continuity}

We now prove that the function $F$ constructed this way is a.s. continuous. To this end, we rely on the fact that $\cutT$ is distributed as a Brownian CRT, comparing as in Lemma \ref{lem:bounded} values of $F$ with distances in $\cutT$.

\begin{prop}
\label{prop:continuity}
Almost surely, $F$ is continuous on $[0,1]$.
\end{prop}

In order to establish the continuity of the function $F$, it is useful to define a ``backward'' construction. Fix $x \in \cutT$. We define a sequence $(B_{i},L_{i})_{0 \leq i < N+1}$ with $N \in \mathbb{Z}_{+} \cup \{+\infty\} $  as follows (we drop the dependence in $x$ to simplify notation). Set $B_{0}=\rho, L_{0}=0$. Then, by induction, if $(B_{i},L_{i})_{0 \leq i \leq k}$ have been constructed, we define $B_{k+1}$ by:
$$\llbracket B_{k}, x \rrbracket \cap \llbracket B_{k},L_{k} \rrbracket= \llbracket B_{k},B_{k+1}\rrbracket, \qquad  L_{k+1}= \begin{cases}  \Lambda \left( \bar{\cC}_{B_{k+1}}^{L_{k}} \right)  &\textrm{ if } B_{k+1} \textrm{ is a branchpoint} \\
B_{k+1} &\textrm{ otherwise}
\end{cases}$$
If $x=B_{k+1}$ we set $N=k+1$ and stop, otherwise we continue.
We say that $(B_{i},L_{i})_{0 \leq i < N+1}$  is the \emph{record sequence} associated with $x$ (see Fig.~\ref{fig:backward} for an illustration).

 \begin{figure}[!ht] \centering
\includegraphics[scale=0.6]{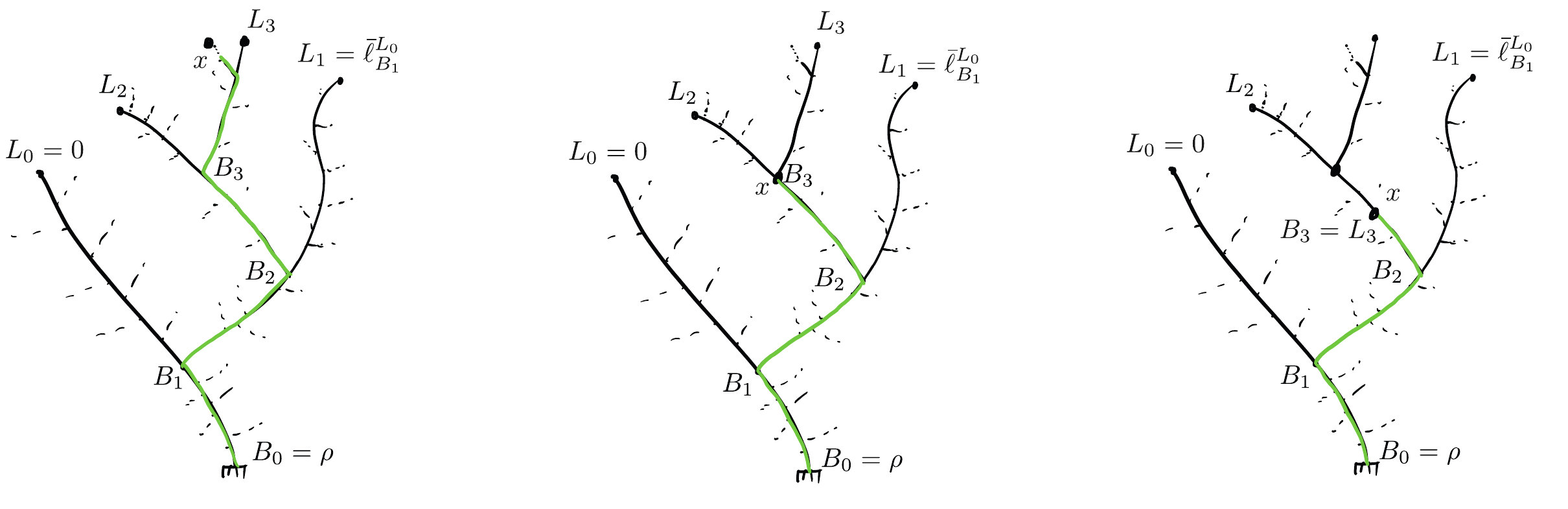}
\caption{Illustration of the backward construction: on the left, $x$ is a leaf and the record sequence associated with $x$ is infinite; in the middle, $x$ is a branchpoint and $(B_{i},L_{i})_{0 \leq i \leq 3}$ is the associated record sequence; on the right, $x$ belongs to the skeleton and $(B_{i},L_{i})_{0 \leq i \leq 3}$ is the associated record sequence.}
\label{fig:backward}
\end{figure}

The next result characterizes points of $\cutT$ whose record sequence is infinite.

\begin{lem}
\label{lem:cvx}
When $N=\infty$, $x$ is a leaf, and $(B_{i})_{i \geq 0}$ converges to $x$.
\end{lem}

\begin{proof}
First, let us show that if $x$ is not a leaf, then $N<\infty$. It suffices to show that $N<\infty$ when $x$ is a branchpoint (this will entail that $N<\infty$ when $x \in \Sk(\cC)$ by considering a descendent of $x$ which is a branchpoint). To this end, recall from Sec.~\ref{ssec:cuttree} that $x$ can be written as $i \wedge j$ for some $i, j \geq 1$.
Let $c \in \llbracket U_{i},U_{j} \rrbracket$ be the cutpoint appearing at time $t_{i,j}$ (observe that $c$ is the first cutpoint falling on  $\llbracket U_{i},U_{j} \rrbracket$).
Let $(c_{p})_{1 \leq p \leq M}$ be the cutpoints falling on $\llbracket \varnothing, U_{i} \rrbracket \cup\llbracket \varnothing, U_{j} \rrbracket $ before time $t_{i,j}$, ordered by their time of appearance (observe that almost surely $M<\infty$, and that all these points fall on $\llbracket \varnothing, U_i \wedge U_j \rrbracket$ except for $c$).
 We construct a subsequence $(c_{p_{k}})_{1 \leq k \leq N}$ of cutpoints precisely corresponding to $(B_{k})_{ 1  \leq i \leq N}$ in $\cC$ as follows. Set $p_{1}=1$. If $c=c_{1}$, we set $N=1$. Otherwise, assuming that $(p_{k})_{1 \leq k \leq n}$ has been defined, set $p_{n+1}= \min \{p>p_{n} : c_{p_{n}} \prec c_{p}\}$; if $c_{p_{n+1}}=c$ we set $N=p_{n+1}$ and stop, otherwise we continue. It is clear that, for all $1 \leq k \leq N$, $c_{p_{k}}=B_k$. Furthermore, since $M<\infty$, this procedure eventually stops, thus showing that $N<\infty$.  

Next, argue by contradiction and assume that $N=\infty$, $x$ is a leaf and $(B_{k})_{k \geq 0}$ converges to a point $u \in \cC$ with $u \neq x$ (observe that $(B_{k})_{k \geq 0}$ always converges as it is increasing). Since $B_{k} \prec x$ for every $k \geq 0$, we have $u \prec x$. Choose a branchpoint $b \in \cC$ such that $u \prec b \prec x$. Then the record sequence associated with $b$ has infinitely many terms, in contradiction with the previous paragraph. 
\end{proof}

When $x$ is a branchpoint (then $N<\infty$),  we set:
\begin{equation}
\label{eq:ellellebar}\ell(x)=L_{N-1}, \qquad \bar{\ell}(x)=L_{N}.
\end{equation}
In terms of the Pac-Man construction, these leaves can be interpreted as follows: when passing at $x$ during its journey, if possible, the Pac-Man eats the subtree above $x$ containing $\ell(x)$ and continues towards $\bar{\ell}(x)$; otherwise it continues towards $\ell(x)$.

We also define for every $x \in \cutT$:
 \begin{equation}
 \label{eq:defh1}h_{1}(x)=\sum_{1 \leq k <N+1} \nu \left(\cutT_{B_{k}}^{L_{k-1}}\right).
 \end{equation}
Then, by Lemma \ref{lem:cvx}, for every $x \in \cC$, if one takes $h=h_{1}(x)$ in the Pac-Man construction,  one precisely gets the sequence $(B_{k},L_{k})_{0 \leq k <N+1}$ with $\pi_{h_{1}}=x$.

The following result is an immediate consequence of the definition:

\begin{lem}
\label{lem:decreasing}
For every branchpoint $B$, $h_{1}$ is decreasing on $\llbracket B, \ell(B) \rrbracket$.
\end{lem}

 When $x$ is a branchpoint, we also set
 \begin{equation}
 \label{eq:defh2}h_{2}(x)=\sum_{1 \leq k  \leq N+1} \nu \left(\cutT_{B_{k}}^{L_{k-1}}\right);
 \end{equation}
observe that if one takes $h=h_{2}(x)$ in the Pac-Man construction, we also get $\pi_{h_2}=x$ (see Fig.~\ref{fig:h1h2} for an illustration). Also note that $h_{2}(x)=h_{1}(x)+\nu( \mathcal{C}^{\bar{\ell}(x)}_{x})$.

 \begin{figure}[!ht] \centering
\includegraphics[scale=0.6]{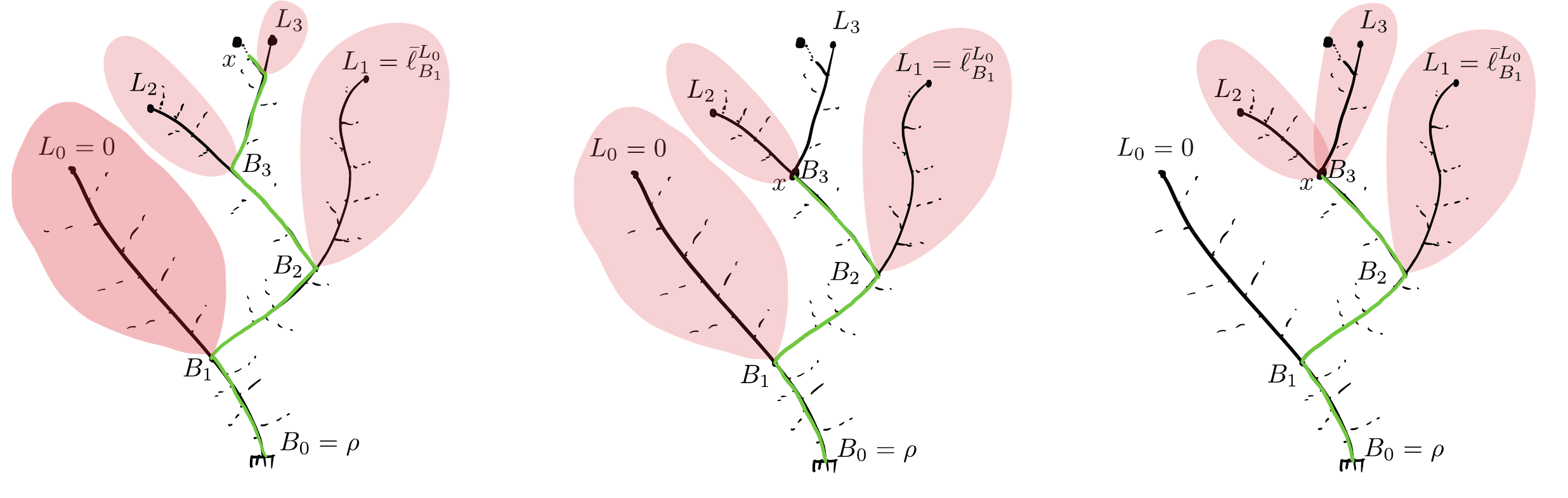}
\caption{Illustration of the definitions of $h_{1}(x)$ and $h_{2}(x)$: on the left, $x$ is a leaf and $h_{1}(x)$ is the sum of the masses of the red subtrees (there are infinitely many of them); in the middle, $x$ is a branchpoint and $h_{1}(x)$ is the sum of the masses of the three red subtrees; on the right, $x$ is the same branchpoint as in the middle and $h_{2}(x)$ is the sum of the masses of the four red subtrees.}
\label{fig:h1h2}
\end{figure}

Finally, we define
\begin{equation}
\label{eq:h2}\mathbb{h}_{2}= \{h_{2}(x): x \in \cB\left(\cC\right) \},
\end{equation}
where we recall that $\cB(\cC)$ denotes the set of all branchpoints of $\cC$, and, to simplify notation, when $x \in \cC$ is not a leaf, we set
\begin{equation}
\label{eq:defh0}
h_{0}(x)=
\begin{cases}
h_{1}(x)-\nu( \mathcal{C}_{x}) & \textrm{ if } x \in \Sk(\cC)\\
h_{1}(x)-\nu\big( \mathcal{C}^{\ell(x)}_{x}\big) & \textrm{ if } x  \in \cB\left(\cC\right).
\end{cases}
\end{equation}

The following result is also an immediate consequence of the definition of the Pac-Man construction, which we record for future use. 

\begin{lem}
\label{lem:target}
Fix $x \in \mathcal{C}$. 
\begin{enumerate}
\item[--] If $x \in \Sk(\cC)$, then for every $h \in (h_{0}(x), h_{1}(x)]$, the final target point $\pi_h$ belongs to $ \mathcal{C}_{x}$. Conversely, each element of $\mathcal{C}_{x}$ is the final target point of an element $h \in (h_{0}(x), h_{1}(x)]$.
\item[--] If $x \in \cB\left(\cC\right)$, then for every $h \in (h_{0}(x), h_{1}(x)]$, the final target point $\pi_h$ belongs to $ \mathcal{C}^{\ell(x)}_{x}$ and for every $h \in [h_{1}(x), h_{2}(x)]$, $\pi_h$ belongs to $\mathcal{C}^{\bar{\ell}(x)}_{x}$. Conversely, every point of $\mathcal{C}^{\ell(x)}_{x}$ is the final target point of some $h \in (h_{0}(x), h_{1}(x)]$, and every point of $\mathcal{C}^{\bar{\ell}(x)}_{x}$ is the final target point of some $h \in [h_{1}(x), h_{2}(x)]$.
\end{enumerate}
\end{lem}

Before proving the continuity of $F$, we  gather some preparatory lemmas.

\begin{lem}
\label{lem:leaf}
Let $\ell \in \mathcal{C}$ be a leaf. Then
$$ \nu( \mathcal{C}_{x}) \cdot \tau_{x}  \quad \mathop{\longrightarrow}_{x \rightarrow \ell} \quad  0.$$
\end{lem}

\begin{proof}
Write
$$ \nu( \mathcal{C}_{x}) \cdot \tau_{x}= \nu( \mathcal{C}_{x}) \cdot \int_{\llbracket \rho,x \rrbracket} \frac{1}{\nu( \mathcal{C}_{z}) } \lambda(\mathrm{d}z)=  \int_{\llbracket \rho, \ell  \rrbracket} \frac{ \nu( \mathcal{C}_{x}) }{\nu( \mathcal{C}_{z}) } \mathbbm{1}_{z \in \llbracket \rho, x \rrbracket } \lambda(\mathrm{d}z).$$
Then observe that the quantity ${ \nu( \mathcal{C}_{x}) }/{\nu( \mathcal{C}_{z}) } \mathbbm{1}_{z \in \llbracket \rho, x \rrbracket }$ is bounded by $1$ and tends to $0$ since $\nu ( \mathcal{C}_{x}) \rightarrow 0$ as $x \rightarrow \ell$. The conclusion follows by dominated convergence. 
\end{proof}

The next lemma compares the difference between two values of $h$ with masses of subtrees in the cut-tree $\cutT$.

\begin{lem}
\label{lem:h1nux}
Take $x \prec y$ in $\mathcal{C}$ with $x \in \Sk(\cC)$  and assume that $\pi_{h'}=y$. Then:
\begin{itemize}
\item[(i)] when $y \in \Sk(\cC)$ we have $|h_{1}(x)-h'| \leq \nu(\cC_{x} \backslash \cC_{y})$.
\item[(ii)] $|h_{1}(x)- h'| \leq \nu(\cC_{x})$;
\end{itemize}
\end{lem}

\begin{proof}
Let $(B^{1}_{i},L^{1}_{i})_{0 \leq i < N_{1}+1}$ and respectively $(B^{2}_{i},L^{2}_{i})_{0 \leq i < N_{2}+1}$ be the record sequences associated with $x$ and $y$. Since $x \prec y$, we have $N_{1} \leq N_{2}$. Also, $x$ is not a leaf, so that $N_{1}<\infty$ and $x= B^{1}_{N_{1}}$.  Observe that we may have $B^{1}_{N_{1}} \neq B^{2}_{N_{1}}$ (e.g.~if $x \in \Sk(\cC)$).  Then $(B^{1}_{i},L^{1}_{i})_{0 \leq i < N_{1}} =(B^{2}_{i},L^{2}_{i})_{0 \leq i < N_{1}}$ and $B_{N_{1}}^{2} \in \llbracket x,L^{1}_{N_{1}-1} \rrbracket$, 
 \begin{figure}[!ht] \centering
\includegraphics[scale=0.6]{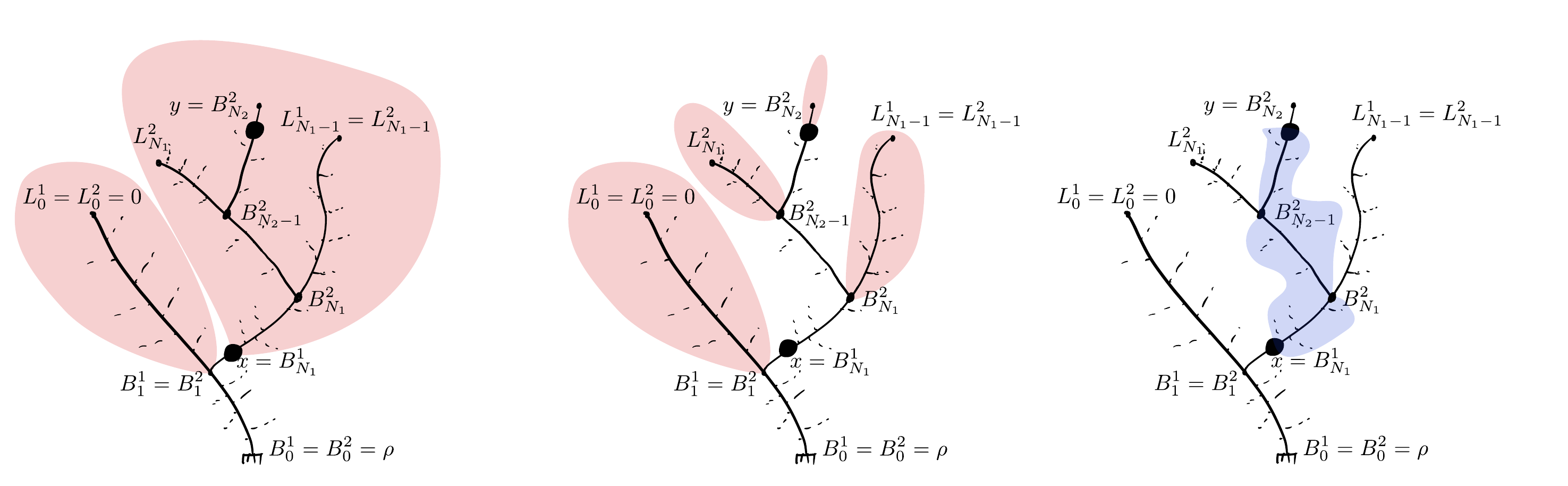}
\caption{Illustration of the proof of Lemma \ref{lem:h1nux} when $x,y \in \Sk(\cC)$. Here $N_{1}=2$ and $N_{2}=4$. Left: the sum of the masses of the two red subtrees is $h_{1}(x)$. Middle: the sum of the masses of the four red subtrees is $h_{1}(y)$. Right: the difference $|h_{1}(x)-h'|$ is at most the mass of the blue subtree $ \mathcal{A}\backslash \mathcal{B}$.}
\label{fig:lemh1h2}
\end{figure}

Recall the notation $\mathbb{h}_{2}$ from \eqref{eq:h2}. Then
$$h_{1}(x)=\sum_{1 \leq k <N_{1}+1} \nu \Big(\cutT_{B^{1}_{k}}^{L^{1}_{k-1}}\Big)\quad \textrm{and} \quad h'=\sum_{1 \leq k <N_{2}+1} \nu \Big(\cutT_{B^{2}_{k}}^{L^{2}_{k-1}}\Big)+\mathbbm{1}_{h' \in \mathbb{h}_{2}} \nu \Big(\cutT_{B^{2}_{N_{2}}}^{L^{2}_{N_{2}}}\Big),$$
so that
$$h_{1}(x)-h'= \sum_{N_{1} \leq k< N_{2}+1} \nu \Big(\cutT_{B^{2}_{k}}^{L^{2}_{k-1}}\Big) +\mathbbm{1}_{h' \in \mathbb{h}_{2}}  \nu \Big(\cutT_{B^{2}_{N_{2}}}^{L^{2}_{N_{2}}}\Big)-\nu \Big(\cutT_{B^{1}_{N_{1}}}^{L^{1}_{N_{1}-1}}\Big).$$ 
Now define
$$ \mathcal{A}= \cutT_{B^{1}_{N_{1}}}^{L^{1}_{N_{1}-1}}, \qquad  \mathcal{B}= \bigcup_{N_{1} \leq k<N_{2}+1} \cutT_{B^{2}_{k}}^{L^{2}_{k-1}}.$$
Observe that the union defining $ \mathcal{B}$ is disjoint. 
When $y \in \Sk(\cC)$ we have $h' \not \in \mathbb{h}_{2}$, $ \mathcal{A} \backslash \mathcal{B} \subset    \cC_{x} \backslash \cC_{y}$. When $y$ is a branchpoint, setting $ \mathcal{B}'= \mathcal{B} \cup \cutT_{B^{2}_{N}}^{L^{2}_{N}}  $, observe that this union is disjoint and  $ \mathcal{A} \backslash \mathcal{B} \subset    \cC_{x}$. The conclusion follows.
\end{proof}

The next lemma bounds from above the difference between two values of $F$.

\begin{lem}
\label{lem:F}
Take $x \prec y$ in $\mathcal{C}$ and $h,h' \in [0,1]$ such that $\pi_{h}=x$ and $\pi_{h'}=y$. Then
$$|F(h)-F(h')| \leq  \tau_{x} |h-h'|+ \dC(x,y).$$
\end{lem}

\begin{proof}
We keep the notation introduced in the beginning of Lemma \ref{lem:h1nux} (in particular it may be helpful to also refer to Fig.~\ref{fig:lemh1h2}): we denote by $(B^{1}_{i},L^{1}_{i})_{0 \leq i < N_{1}+1}$ and respectively $(B^{2}_{i},L^{2}_{i})_{0 \leq i < N_{2}+1}$ the record sequences associated with $x$ and $y$, so that $N_{1} \leq N_{2}$, $x$ is not a leaf,  $N_{1}<\infty$,  $x= B^{1}_{N_{1}}$, $(B^{1}_{i},L^{1}_{i})_{0 \leq i < N_{1}} =(B^{2}_{i},L^{2}_{i})_{0 \leq i < N_{1}}$ and $B_{N_{1}}^{2} \in \llbracket x,L^{1}_{N_{1}-1} \rrbracket$. Recall the definitions of $F$ in \eqref{eq:defF} and of $h_{1},h_{2}$ in \eqref{eq:defh1}, \eqref{eq:defh2}. 

We have
$$h=\sum_{1 \leq k <N_{1}+1} \nu \Big(\cutT_{B^{1}_{k}}^{L^{1}_{k-1}}\Big)+\mathbbm{1}_{h \in \mathbb{h}_{2}}  \nu\Big(\cutT_{B^{1}_{N_1}}^{L^{1}_{N_1}}\Big), \qquad F(h)=\sum_{1 \leq k <N_{1}+1} \tau_{B^{1}_{k}}\nu \Big(\cutT_{B^{1}_{k}}^{L^{1}_{k-1}}\Big)+\mathbbm{1}_{h \in \mathbb{h}_{2}} \tau_{B^{1}_{N_{1}}} \nu\Big(\cutT_{B^{1}_{N_{1}}}^{L^{1}_{N_{1}}}\Big) $$
and
$$h'=\sum_{1 \leq k <N_{2}+1} \nu \Big(\cutT_{B^{2}_{k}}^{L^{2}_{k-1}}\Big)+\mathbbm{1}_{h' \in \mathbb{h}_{2}} \nu \Big(\cutT_{B^{2}_{N_2}}^{L^{2}_{N_2}}\Big), \qquad F(h')=\sum_{1 \leq k <N_{2}+1} \tau_{B^{2}_{k}}\nu \Big(\cutT_{B^{2}_{k}}^{L^{2}_{k-1}}\Big)+\mathbbm{1}_{h' \in \mathbb{h}_{2}} \tau_{B^{2}_{N_{2}}} \nu\Big(\cutT_{B^{2}_{N_{2}}}^{L^{2}_{N_{2}}}\Big).$$
Thus, setting $$m_{k}=\nu \Big(\cutT_{B^{2}_{k}}^{L^{2}_{k-1}}\Big) \textrm{ for } k<N_{2}, \qquad m_{N_{2}}=\nu\Big(\cutT_{B^{2}_{N_{2}}}^{L^{2}_{N_{2}-1}}\Big)+\mathbbm{1}_{h' \in \mathbb{h}_{2}} \nu\Big(\cutT_{B^{2}_{N_{2}}}^{L^{2}_{N_{2}}}\Big)$$
with the convention that $m_{N_2}=0$ if $N_2=\infty$ and remembering that $x= B^{1}_{N_{1}}$, we have
$$F(h)-F(h')= \tau_{x}\nu \Big(\cutT_{B^{1}_{N_{1}}}^{L^{1}_{N_{1}-1}}\Big)+\mathbbm{1}_{h \in \mathbb{h}_{2}} \tau_{x} \nu\Big(\cutT_{B^{1}_{N_{1}}}^{L^{1}_{N_{1}}}\Big) -\sum_{N_{1} \leq k <N_{2}+1} \tau_{B^{2}_{k}} m_{k}$$
and
$$h-h'= \nu \Big(\cutT_{B^{1}_{N_{1}}}^{L^{1}_{N_{1}-1}}\Big)+\mathbbm{1}_{h \in \mathbb{h}_{2}}  \nu\Big(\cutT_{B^{1}_{N_{1}}}^{L^{1}_{N_{1}}}\Big) -\sum_{N_{1} \leq k <N_{2}+1} m_{k}.$$
It follows that
\begin{equation}
\label{eq:FhFhprime}
F(h)-F(h')=  \tau_{x} (h-h')+ \sum_{N_{1} \leq k <N_{2}+1}  (  \tau_{x} - \tau_{B^{2}_{k}}) m_{k}.
\end{equation}
In particular,
$$|F(h)-F(h')| \leq  \tau_{x} |h-h'|+ \sum_{N_{1} \leq k <N_{2}+1}  (  \tau_{B^{2}_{k}}- \tau_{x} ) m_{k}.$$

To control the sum, we perform an Abel transformation by setting $\sigma_{k}=\tau_{B^{2}_{k}}-\tau_{B^{2}_{k-1}}$ for $ N_{1}<k<N_{2}+1$ and $\sigma_{N_{1}}=\tau_{B^{2}_{N_{1}}}-\tau_{x}$. Then 
$$
\sum_{N_{1}  \leq  k <N_{2}+1} (\tau_{B^{2}_{k}}-\tau_{x}) \cdot m_{k}= \sum_{N_{1}  \leq  k <N_{2}+1} \sum_{ i \in \llbracket N_1,k \rrbracket} \sigma_{i}   \cdot m_{k}= \sum_{ N_{1} \leq i < N_{2}+1}  \sigma_{i} \sum_{k \in \llbracket i, N_2+1 \llbracket}  m_{k}.
$$
Then observe that, as in the proof of Lemma \ref{lem:h1nux},  for every $ N_{1} \leq i <N_{2}+1$ we have $ \sum_{k \in \llbracket i, N_2+1 \llbracket} m_{k} \leq \nu\big( \mathcal{C}_{B^{2}_{i}}\big)$. Also,  $\sigma_{N_{1}} \leq  \dC(x,B^{2}_{N_{1}})/ \nu( \mathcal{C}_{B^{2}_{N_{1}}} )$ and $\sigma_{i} \leq  \dC(B^{2}_{i-1},B^{2}_{i}) / \nu( \mathcal{C}_{B^{2}_{i}})$ for $N_{1} <i<N_{2}+1$. The conclusion follows.
\end{proof}

We can now prove the continuity of the function $F$.

\begin{proof}[Proof of Proposition \ref{prop:continuity}]
Fix $h \in [0,1]$. We want to prove that $F$ is continuous at $h$. We distinguish several cases according to the nature of the final target point $z := \pi_h$.

\emph{$\star$ First case:} $z$ is a leaf (see Fig.~\ref{fig:case-1-2}, left). Fix $\varepsilon>0$. Using in particular Lemma \ref{lem:leaf} and the fact that $\cC$ is a Brownian CRT, we may choose $x \in \Sk(\cC)$ such that  $x \prec z$ and $\tau_{x} \, \nu( \cC_{x}) \leq \varepsilon$, $\nu(\cC_{x}) \leq \varepsilon$ and $\mathsf{Diam}(\cC_{x})<\varepsilon$.  Observe that  $h_{1}(x)-\nu(\cC_{x})<h<h_{1}(x)$ by Lemma \ref{lem:target}. Then, by Lemmas \ref{lem:h1nux} (ii) and \ref{lem:F}, we have
$$|F(h_{1}(x))-F(h)| \leq \tau_{x} |h_{1}(x)-h|+ \dC(x,z)\leq  \tau_{x} \nu(\cC_{x}) + \varepsilon \leq 2 \varepsilon.$$ Next,  take $ h' \in (h_{1}(x)-\nu(\cC_{x}),h_{1}(x))$.
By Lemma \ref{lem:target}, there exists $y \in \cC_{x}$ such that $\pi_{h'}=y$. Then, as before, again by Lemmas \ref{lem:h1nux} (ii) and \ref{lem:F}, we have $$|F(h_{1}(x))-F(h')| \leq \tau_{x} |h_{1}(x)-h'|+ \dC(x,z)\leq  \tau_{x} \nu(\cC_{x}) + \varepsilon \leq 2 \varepsilon.$$ We conclude that $|F(h)-F(h')| \leq 4 \varepsilon$.

\emph{$\star$ Second case:} $z \in \Sk(\cC)$   (see Fig.~\ref{fig:case-1-2}, right). Fix $\epsilon>0$. Let $(B_{i},L_{i})_{0 \leq i < N+1}$  be the record sequence associated with $z$, so that $z=B_{N}$ and $z \in \rrbracket B_{N-1}, L_{N-1}\llbracket$. Fix $z' \in \cC$ such that $z \prec z' \prec L_{N-1}$. Take $\varepsilon>0$.
Then choose $u,v \in \Sk(\cC)$  such that $B_{N-1} \prec u \prec z \prec v  \prec z'$ such that $\nu( \cC_{u} \backslash \cC_{v}) \leq \varepsilon/ \tau_{z'}$ and $\mathsf{Diam}( \cC_{u} \backslash \cC_{v}) \leq \varepsilon$.
Then, by Lemmas \ref{lem:h1nux} (i) and \ref{lem:F}, we have
$$|F(h_{1}(u))-F(h)| \leq \tau_{u} |h_{1}(u)-h|+ \dC(u,z)\leq  \tau_{z'} \nu ( \cC_{u} \backslash \cC_{z})+\varepsilon \leq  \tau_{z'} \nu ( \cC_{u} \backslash \cC_{v})+\varepsilon \leq 2 \varepsilon.$$
Next, observe that by Lemma \ref{lem:decreasing} we have $h_{1}(v)<h<h_{1}(u)$. Take $h' \in (h_{1}(v),h_{1}(u))$. By Lemma \ref{lem:target}, since $v \in \mathcal{C}_{u}$, we have $h_{1}(u)-\nu (\mathcal{C}_{u})<h_{1}(v)$, so we have $ h_{1}(u)-\nu( \mathcal{C}_{u})< h' < h_{1}(u)$.

 By Lemma \ref{lem:target}, there exists $y \in \cC_{u}$ such that $\pi_{h'}=y$. Then, again by Lemma \ref{lem:F}, we have
$$|F(h_{1}(u))-F(h')| \leq \tau_{u} |h_{1}(u)-h'|+ \dC(u,y)\leq  \tau_{z'}(h_{1}(u)-h_{1}(v))+\varepsilon \leq  \tau_{z'} \nu ( \cC_{u} \backslash \cC_{v})+\varepsilon \leq 2 \varepsilon$$
where we have also used Lemma \ref{lem:h1nux} (i) to write $h_{1}(u)-h_{1}(v) \leq   \nu ( \cC_{u} \backslash \cC_{v})$.
We conclude that $|F(h)-F(h')| \leq 4 \varepsilon$.

 \begin{figure}[!ht] \centering
\includegraphics[scale=0.8]{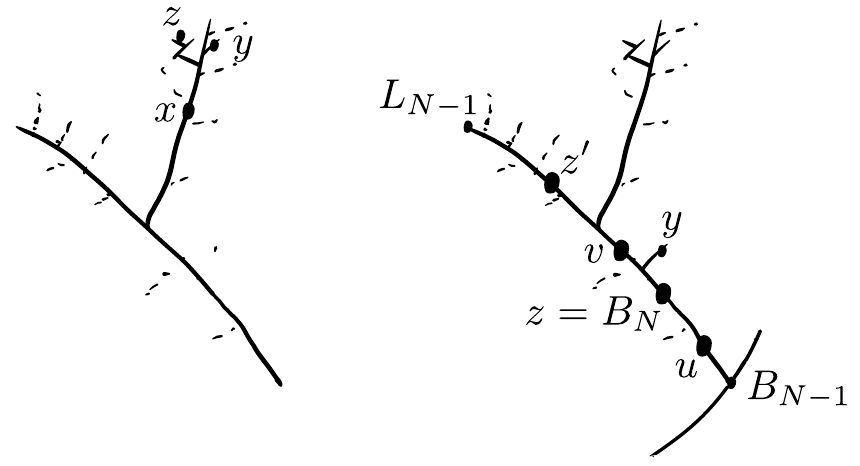}
\caption{Illustration of the first case when $z$ is a leaf (left) and the second case  where $z \in \Sk(\cC)$  (right).}
\label{fig:case-1-2}
\end{figure}

\emph{$\star$ Third case:} $z$ is a branchpoint. Let $(B_{i},L_{i})_{0 \leq i < N+1}$  be the record sequence associated with $z$ with $N<\infty$.  We consider two subcases.

\emph{$\star$ $\star$ First subcase:} $h$ is of the form $h=h_{1}(z)$ (see Fig.~\ref{fig:case-3-1}). We first show that $F$ is right-continuous at $h$. Fix $\epsilon>0$. Choose a point $ u \in \Sk(\cC)$ such that $z \prec u \prec L_{N}$ and  $\nu( \cC_{u}) \leq \varepsilon/ \tau_{u}$ and $\mathsf{Diam}(\cC_{u})<\varepsilon$. By definition of the Pac-Man construction, $h_{1}(u)=h+\nu(\cC_{u})$ and $F(h_{1}(u))=F(h)+ \tau_{u} \nu( \cC_{u})$. In particular, $|F(h_{1}(u))-F(h)| \leq \varepsilon$.

Now take $h' \in (h,h+\nu( \mathcal{C}_{u}))$. Since $h_{1}(u)-\nu( \cC_{u})<h'<h_{1}(u)$, by Lemma \ref{lem:target}, there  exists $y \in \cC_{u}$ such that $\pi_{h'}=y$. Then, as before,  by Lemma \ref{lem:F}, we have 
$$|F(h_{1}(u))-F(h')| \leq \tau_{y}|h_{1}(u)-h'|+\dC(u,y)\leq  \tau_{u}\nu(\cC_{u})+\varepsilon \leq 2 \varepsilon.$$
We conclude that $|F(h)-F(h')| \leq 3 \varepsilon$.

Let us next show that $F$ is left-continuous at $h$. Fix a point $v \in \Sk(\cC)$ such that  $z \prec v \prec L_{N-1}$. Take $\varepsilon>0$ and choose a point $ u \in \Sk(\cC)$  such that $z \prec u \prec v$ and  $\nu( \cC^{L_{N-1}}_{z} \backslash \cC_{u}) \leq \varepsilon/ \tau_{v}$ and $\mathsf{Diam}(\cC^{L_{N-1}}_{z} \backslash\cC_{u})<\varepsilon$. Observe that by definition of the Pac-Man construction, $h_{1}(u)+ \nu( \cC^{L_{N-1}}_{z} \backslash \cC_{u})=h$. Take $h' \in (h_{1}(u),h)$. By Lemma \ref{lem:target}, there exists $y \in \cC^{L_{N-1}}_{z} \backslash \cC_{u}$ such that $\pi_{h'}=y$. Then, as before,  by Lemma \ref{lem:F}, we have 
$$|F(h)-F(h')| \leq \tau_{z}|h-h'|+\dC(z,y)\leq  \tau_{v} \nu( \cC^{L_{N-1}}_{z} \backslash \cC_{u})  +\varepsilon \leq 2 \varepsilon.$$

 \begin{figure}[!ht] \centering
\includegraphics[scale=0.8]{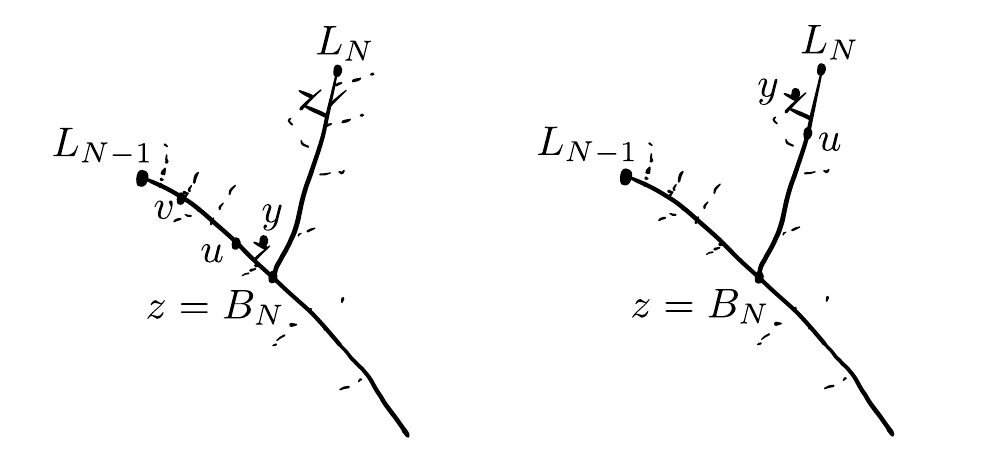}
\caption{Illustration of the case $h=h_{1}(z)$. Left: proof of the left-continuity; right: proof of the right-continuity.
}
\label{fig:case-3-1}
\end{figure}

\emph{$\star$ $\star$ Second subcase:}  $h$ is of the form $h=h_{2}(z)$ (see Fig.~\ref{fig:case-3-2}).

Let us first show that $F$ is left-continuous at $h$. Fix a point $v \in \Sk(\cC)$ such that  $z \prec v \prec L_{N}$. Take $\varepsilon>0$ and choose a point $ u \in \Sk(\cC)$ such that $z \prec u \prec v$,  $\nu( \cC^{L_{N}}_{z} \backslash \cC_{u}) \leq \varepsilon/ \tau_{v}$ and $\mathsf{Diam}(\cC^{L_{N}}_{z} \backslash\cC_{u})<\varepsilon$. Observe that by definition of the Pac-Man construction, $h_{1}(u)+ \nu( \cC^{L_{N-1}}_{z} \backslash \cC_{u})=h$. Take $h' \in (h_{1}(u),h)$. By Lemma \ref{lem:target}, there exists $y \in \cC^{L_{N-1}}_{z} \backslash \cC_{u}$ such that $\pi_{h'}=y$. Then, as before,  by Lemma \ref{lem:F}, we have 
$$|F(h)-F(h')| \leq \tau_{z}|h_{1}(z)-h'|+\dC(z,y)\leq  \tau_{v} \nu( \cC^{L_{N-1}}_{z} \backslash \cC_{u})  +\varepsilon \leq 2 \varepsilon.$$
 
 \begin{figure}[!ht] \centering
\includegraphics[scale=0.8]{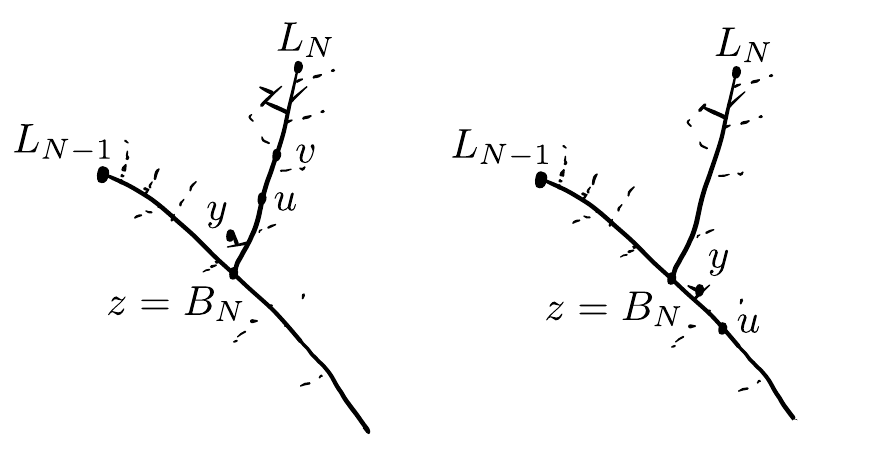}
\caption{Illustration of the case $h=h_{2}(z)$. Left: proof of the left-continuity; right: proof of the right-continuity.}
\label{fig:case-3-2}
\end{figure}

Let us next show that $F$ is right-continuous at $h$. Take $\varepsilon>0$ and fix a point $u \in \Sk(\cC)$  such that $B_{N-1} \prec u \prec z$,    $\nu( \cC_{u} \backslash \cC_{z}) \leq \varepsilon/ \tau_{z}$ and $\mathsf{Diam}(\cC_{u} \backslash\cC_{z})<\varepsilon$.  Observe that by definition of the Pac-Man construction, $h_{1}(u)=h+\nu( \cC_{u} \backslash \cC_{z}) $. Take $h' \in (h,h_{1}(u))$. By Lemma \ref{lem:target}, there exists $y \in\cC_{u} \backslash \cC_{z}$ such that $\pi_{h'}=y$. Then, as before,  by Lemma \ref{lem:F}, we have 
$$|F(h)-F(h')| \leq \tau_{u}|h-h'|+\dC(u,y)\leq  \tau_{z} \nu( \cC_{u} \backslash \cC_{z})  +\varepsilon \leq 2 \varepsilon.$$
This completes the proof.
\end{proof}

\subsection{The function $F$ codes the Aldous-Pitman fragmentation of the CRT}

We have constructed a continuous excursion-type function $F$ from the Aldous-Pitman fragmentation of the Brownian CRT $\cT$. In order to prove Theorem \ref{thm:couplage} (i), before showing that $F$ is in law the Brownian excursion, we  first show  the following result
\begin{prop}
\label{prop:tailles}
A.s. for every $t \geq 0$, the nonincreasing rearrangement of the masses of the connected components of $ \mathcal{T} \backslash \cP_{t}$ is the same as the nonincreasing rearrangement of the lengths of the excursions of $(F(h)-th)_{ 0 \leq h \leq 1}$ above its running infimum.\end{prop}

 To simplify notation, set $F_{t}(h)=F(h)-th$ for $0 \leq h \leq 1$.

\begin{lem}
\label{lem:Ft}
Take $x \prec y$ in $\mathcal{C}$ and $h,h' \in [0,1]$ such that $\pi_{h}=x$ and $\pi_{h'}=y$. Then $F_{\tau_{x}}(h')> F_{\tau_{x}}(h)$.
\end{lem}

\begin{proof}
This readily follows from the  identity \eqref{eq:FhFhprime} appearing in the proof of Lemma \ref{lem:F}. Indeed, keeping the same notation, if $x=B^{1}_{N_{1}}=B^{2}_{N_{1}}$ then $N_{2}>N_{1}$ and $\tau_{x}<\tau_{B^{2}_{k}}$  for every $N_{1} < k < N_{2}+1$ since $x$ is an ancestor of $B^{2}_{k}$. If $x=B^{1}_{N_{1}} \neq B^{2}_{N_{1}}$ then similarly $\tau_{x}<\tau_{B^{2}_{N_{1}}}$.
\end{proof}

\begin{proof}[Proof of Proposition \ref{prop:tailles}.] 
By Lemma \ref{lem:bijectionsubtreescomponents}, for every $t \geq 0$, the connected components of $ \mathcal{T} \backslash \cP_{t}$ are in bijection with subtrees of $ \mathcal{C}$ of the form $\mathcal{C}_{x}^{\ell(x)}$ or $ \mathcal{C}_{x}^{\bar{\ell}(x)}$ for $x \in \mathcal{C}$ with $\tau_{x}=t$; and this bijection conserves the masses. If $C$ is a connected component, recall that we denote by $\Phi(C)$ the corresponding subtree of  $\mathcal{C}$ (in particular $\nu(\Phi(C))=\mu(C)$). 

Let $C$ be a connected component of  $ \mathcal{T} \backslash \cP_{t}$ and $x \in \mathcal{C}$ with $\tau_{x}=t$. Observe that then $x$ is not a leaf; we denote by $(B_{i},L_{i})_{0 \leq i \leq N}$ the record sequence of $x$.  Recall from \eqref{eq:defh0} the definition of $h_{0}(x)$. 

We claim that $ F_{t}(h_{0}(x))=F_{t}(h_{1}(x))$ and that
\begin{equation}
\label{eq:claims1}\forall h \in (h_{0}(x),h_{1}(x)), \quad F_{t}(h)>F_{t}(h_{1}(x)); \qquad  \forall h \in (0,h_{0}(x)),  \quad F_{t}(h)> F_{t}(h_{1}(x)).
\end{equation}
This implies that when $x \in \Sk(\cC)$, we have an excursion of length $\nu( \mathcal{C}_{x})$ of $F_{t}$ above its running infimum, and when $x$ is a branchpoint we have an excursion of length $\nu( \cC_{x}^{\ell(x)})$ of $F_{t}$ above its running infimum.

To check that  $ F_{t}(h_{0}(x))=F_{t}(h_{1}(x))$, observe that by definition
$$x=B_{N}, \quad t=\tau_{B_{N}}, \quad h_{1}(x)=\sum_{1 \leq k  \leq N} \nu \left(\cutT_{B_{k}}^{L_{k-1}}\right), \quad F(h_{1}(x))= \sum_{1 \leq  k \leq N } \tau_{B_{k}} \cdot \nu \left(\cutT_{B_{k}}^{L_{k-1}}\right),$$
so
$$F_{t}(h_{1}(x))=\sum_{1 \leq  k \leq N } \tau_{B_{k}} \cdot \nu \left(\cutT_{B_{k}}^{L_{k-1}}\right) - \tau_{B_{N}}  \sum_{1 \leq  k \leq N }  \nu \left(\cutT_{B_{k}}^{L_{k-1}}\right).$$
In addition, observe that
$$h_0(x)= \sum_{1 \leq  k \leq N-1 }  \nu \left(\cutT_{B_{k}}^{L_{k-1}}\right)=h_1(B_{N-1}).$$
It follows that $\pi_{h_0(x)}=B_{N-1}$, and
$$
F(h_1(B_{N-1}))= \sum_{1 \leq  k \leq N-1 } \tau_{B_{k}} \cdot \nu \left(\cutT_{B_{k}}^{L_{k-1}}\right).
$$
Hence
$$F_{t}(h_{0}(x))=F_{t}(h_{1}(B_{N-1}))=\sum_{1 \leq  k \leq N-1 } \tau_{B_{k}} \cdot \nu \left(\cutT_{B_{k}}^{L_{k-1}}\right) - \tau_{B_{N}}  \sum_{1 \leq  k \leq N-1 }  \nu \left(\cutT_{B_{k}}^{L_{k-1}}\right) = F_{t}(h_1(x)).$$

The first inequality in \eqref{eq:claims1} readily follows  from Lemma \ref{lem:Ft}, since the final target point $\pi_h$ of any element $h \in (h_{0}(x),h_{1}(x))$ belongs to $\cutT_{x} \backslash \{x\} $ by Lemma \ref{lem:target}.

To establish  the second inequality in \eqref{eq:claims1}, take   $h \in (0,h_{0}(x))$.  Let $1 \leq m \leq N-1$ be such that 
$$\sum_{1 \leq  k \leq m-1 }  \nu \left(\cutT_{B_{k}}^{L_{k-1}}\right) \leq h < \sum_{1 \leq  k \leq m }  \nu \left(\cutT_{B_{k}}^{L_{k-1}}\right)=h_{1}(B_{m}).$$
Then $\pi_{h} \in \cutT_{B_{m}}^{L_{m-1}}$, so that by Lemma \ref{lem:Ft} we have $F_{\tau_{B_{m}}}(h) \geq F_{\tau_{B_{m}}}(h_{1}(B_m))$  (with the inequality being strict if $\pi_{h} \neq B_m$). Since $t \geq \tau_{B_{m}}$, this entails $F_{t}(h) \geq F_{t}(h_{1}(B_{m}))$. It remains to note that
$$ F_{t}(h_{1}(B_{m})) \geq F_{t}(h_{1}(x))$$
with the inequality being strict if $m<N-1$.
Indeed,
$$F_{t}(h_{1}(B_{m}))=\sum_{1 \leq  k \leq m } \tau_{B_{k}} \cdot \nu \left(\cutT_{B_{k}}^{L_{k-1}}\right) - \tau_{B_{N}}  \sum_{1 \leq  k \leq m }  \nu \left(\cutT_{B_{k}}^{L_{k-1}}\right)$$
so
$$ F_{t}(h_{1}(B_{m})) - F_{t}(h_{1}(x))= \tau_{B_{N}}  \sum_{m+1 \leq  k \leq N-1 }  \nu \left(\cutT_{B_{k}}^{L_{k-1}}\right)-\sum_{m+1 \leq  k \leq N-1 } \tau_{B_{k}} \cdot \nu \left(\cutT_{B_{k}}^{L_{k-1}}\right),$$
which entails the result since $\tau_{B_{N}} > \tau_{B_{k}}$ for $m+1 \leq k \leq N-1$. Since one of the two inequalities $F_{t}(h) \geq F_{t}(h_{1}(B_{m}))$ or $ F_{t}(h_{1}(B_{m})) \geq F_{t}(h_{1}(x))$ is strict (because $h<h_{0}(x)$ so $\pi_{h} \neq B_{N-1}$), the second inequality in \eqref{eq:claims1} follows.

To finish the proof, assume that $x$ is a branchpoint.  We claim that $F_{t}(h_{1}(x))=F_{t}(h_{2}(x))$ and that
\begin{equation}
\label{eq:claims2} \forall h \in (h_{1}(x),h_{2}(x)), \qquad  F_{t}(h)> F_{t}(h_{1}(x)).
\end{equation}
This implies that  we have an excursion of length $\nu\big( \cC_{x}^{\bar{\ell}(x)}\big)$ of $F_{t}$ above its running infimum.

The fact that $F_{t}(h_{1}(x))=F_{t}(h_{2}(x))$ is proved exactly in the same way as the identity $ F_{t}(h_{0}(x))=F_{t}(h_{1}(x))$, by using the definition of $F_{t}$. 
Finally, \eqref{eq:claims2}  follows from Lemma \ref{lem:Ft} by observing that when $h \in (h_{1}(x),h_{2}(x))$, $\pi_{h}$ belongs to $ \mathcal{C}_{x} \backslash \{x\} $.
\end{proof}

We establish the following result for future use.

\begin{lem}
\label{lem:measure}
The probability measure $\nu$ on $\cC$ is the push-forward of the Lebesgue measure on $[0,1]$ by $h \mapsto \pi_h$.
\end{lem}

\begin{proof}
Recall that for $x \in \cC$  the tree $\cC_{x}$  is the set of all (weak) descendents of $x$ in $\cC$. First, observe that $ \{ \cC_{x} :x \in \Sk(\cC) \} $ is a generating $\pi$-system of $ \cC$. Hence, if two probability measures $\nu$ and $\tilde{\nu}$ supported on the set of leaves of $\cC$ satisfy
$$\forall x \in \Sk\left(\cC\right), \qquad \nu \left( \cC_{x}\right)=\tilde{\nu} \left( \cC_{x}\right),$$
then $\nu=\tilde{\nu}$.

Now, let $\tilde{\nu}$ be the push-forward of the Lebesgue measure $Leb$ on $[0,1]$ by $\pi$. For $x \in \Sk(\cC)$, keeping the notation of \eqref{eq:defh0}, by Lemma \ref{lem:target} we have
$$\tilde{\nu} \left( \cC_{x}\right)=Leb \{ x \in [0,1]: \pi_{x} \in \cC_{x}\}=Leb([h_{0}(x),h_{1}(x)])=\nu(\cC_{x}),$$
and the desired result follows.
\end{proof}

\begin{rk}
\label{rk:rem1}
By \cite[Proposition 12]{ADG19}, $\{\rho\} \cup \N$ are i.i.d.~with law $\nu$ and conditionally  given $(\cC, \nu,\N)$, for $b \in \cB(\cC)$,  $\bar{\ell}(b)$ has law $\nu_{\cC_{b}^{\bar{\ell}(b)}}$ and these random variables are independent. Indeed, in the notation of the latter reference, we have $Z^{y}_{b}=\bar{\ell}(b)$ for $y \in \cC_{b}^{\bar{\ell}(b)}$ since by definition $\bar{\ell}(b)$ is the image of $b$ in $\nu_{\cC_{b}^{\bar{\ell}(b)}}$.

\end{rk}

\subsection{The function $F$ is a Brownian excursion}
\label{ssec:exc}

In order to prove that $F$ is distributed as a Brownian excursion $\be$, we show that both $F$ and $\be$ satisfy a same recursive equation, which has a unique solution in distribution. 

For every continuous function $f: [0,1] \rightarrow \R_{+}$, define $(P_t(f), t \geq 0)$ as follows:
$$\forall t>0, \qquad P_t(f) = \inf \left\{ u>0, f(u)-tu=0 \right\}.$$

\begin{prop}
\label{prop:charaterization}
Let $f: [0,1] \rightarrow \R_+$ be a (random) continuous function satisfying the following:
\begin{itemize}[noitemsep,nolistsep]
\item[(i)]  $f(0)=0$ a.s.
\item[(ii)] The two processes $(P_t(f))_{t \geq 0}$ and $ \left( \frac{1}{1+S_t} \right)_{t \geq 0}$ have the same law, where  $S$ is a $1/2$-stable subordinator with Laplace exponent $\Phi(\lambda)=(2\lambda)^{1/2}$.
\item[(iii)] Define $(t_i,  a_i, b_i)_{i \geq 1}$ as follows: $\{ t_i, i \geq 1 \}$ is the set of jump times of $(P_t(f), t \geq 0)$ and, for all $i \geq 1$, $(a_i, b_i) = \left( P_{t_i}(f), P_{t_i-}(f) \right)$. Furthermore, the $t_i$'s are sorted in nonincreasing values of $b_i-a_i$. Then, conditionally given $(P_t(f))_{t \geq 0}$, $\left\{ \left((b_i-a_i)^{-1/2} \left( f\left(a_i + (b_i-a_i)u\right)-t_i\left( a_i+(b_i-a_i) u \right) \right)\right)_{0 \leq u \leq 1} \right\}_{i \geq 1}$ are i.i.d. random variables distributed as $(f(u), 0 \leq u \leq 1)$.
\end{itemize}

Then $f$ and $\be$ have the same law.\end{prop}

In particular, observe that if $f$ satisfies these assumptions, then $f(1)=0$ a.s.

\begin{lem}
\label{lem:excbrowprop}
The standard Brownian excursion $\be$ on $[0,1]$ satisfies (i), (ii), (iii).
\end{lem}

In order to prove that $\be$ satisfies these properties, we need the following result from Chassaing and Janson \cite{CJ01}.

\begin{thm}{\cite[Theorem $2.6$]{CJ01}}
Consider $b$ and $\be$, respectively a Brownian bridge and a Brownian excursion on $[0,1]$, extended on $\R$ so that they are $1$-periodic.
For all $a \geq 0$, define the following two processes:
\begin{itemize}[noitemsep,nolistsep]
\item $X_a$ the reflected Brownian bridge $|b|$ conditioned to have local time at time $1$ at $0$ equal to $a$;
\item $Z_a = \Psi_a \be$, where $\Psi_a f(t) = f(t)-at-\inf_{-\infty<s \leq t} \{f(s)-as \}$.
\end{itemize}
Denote by $L_t(X_a)$ the local time of $X_a$ up to time $t$, and $V \in [0,1] $ the unique point at which $t \mapsto L_t(X_a)-at$ is maximum. Then,
\begin{align*}
Z_a \overset{(d)}{=} X_a(V+\cdot). 
\end{align*}
\end{thm}

Let us show Lemma \ref{lem:excbrowprop}.

\begin{proof}[Proof of Lemma \ref{lem:excbrowprop}]
First observe that it is clear by the Markov property and \cite[Proposition $11$]{Ber00} that $\be$ satisfies (i) and (ii). Let us prove that it satisfies (iii). To this end, observe that almost surely $X_{a}(V)=0$ by definition of $V$. Hence, the excursions of $Z_a$, ordered by non-increasing order of length, are distributed as the excursions, ordered by non-increasing order of length, of $X_{a}$. As a consequence, conditionally given their endpoints, the excursions  of $X_{a}$ ordered by non-increasing order of length, are independent and distributed as appropriately rescaled standard Brownian excursions (see e.g. the proof of \cite[Lemma $12$]{Pit99}). This proves that $\be$ satisfies (iii).
\end{proof}

Let us now prove Proposition \ref{prop:charaterization}.

\begin{proof}[Proof of Proposition \ref{prop:charaterization}]
It thus remains to check that these assumptions characterize the distribution of $f$. Let $f$ be a function satisfying (i), (ii), (iii). For every $\varepsilon>0$ we shall construct a coupling $(f,\be)$ such that $\P( \|f- \be \|>\varepsilon)<\varepsilon$, which will imply that $f$ and $\be$ have the same law (indeed, this implies e.g.~that the Lévy-Prokhorov distance between the laws of $f$ and $\be$ is at most $\varepsilon$).  To simplify notation, set
$$\mathcal{U}=\bigcup_{k \in \Z_+} \N^k.$$
If $\mathbf{s} \in \N^k$ with $k \geq 0$ we set $|\mathbf{s}|=k$.
Consider a family of i.i.d. $1/2$-stable subordinators $(S^{\mathbf{s}})_{\mathbf{s} \in \mathcal{U}}$ with the law of $S$, where by convention $\N^0=\{ \varnothing \}$. For each $\mathbf{s} \in \mathcal{U} $, set $Z^\mathbf{s}_t = (1+S^\mathbf{s}_t)^{-1}$ for $t \geq 0$. Denote by $\mathbb{T}^\mathbf{s} \coloneqq (t_i^\mathbf{s})_{i \geq 1}$ the set of jump times of $S^\mathbf{s}$, ordered by decreasing values of $d_i^\mathbf{s}-g_i^\mathbf{s}$, where $d_i^\mathbf{s}=Z^\mathbf{s}_{t_i-}$ and $g_i^\mathbf{s}=Z^\mathbf{s}_{t_i}$ (in case of equality, sort them by increasing value of $g_i^\mathbf{s}$).

Now we define by induction sets of points $(\mathcal{R}^\mathbf{s})_{\mathbf{s} \in \mathcal{U}}$ as follows. Define intervals $[a^{\mathbf{s}}, b^\mathbf{s}]_{\mathbf{s} \in \mathcal{U}}$ as: 
\begin{itemize}[noitemsep,nolistsep]
\item for $\mathbf{s}=\varnothing$, $a^\varnothing=0, b^\varnothing=1$;
\item for $\mathbf{s} \neq \varnothing$, let $\overline{\mathbf{s}} \in \mathcal{U}, i \in \N$ be such that $\mathbf{s}=\overline{\mathbf{s}} \cdot i$. Then, we define 
\begin{align*}
a^\mathbf{s} \coloneqq a^{\overline{\mathbf{s}}} + \left(b^{\overline{\mathbf{s}}} - a^{\overline{\mathbf{s}}}\right) g_i^\mathbf{s},  \qquad b^\mathbf{s} \coloneqq a^{\overline{\mathbf{s}}} + \left(b^{\overline{\mathbf{s}}} - a^{\overline{\mathbf{s}}}\right) d_i^\mathbf{s}.
\end{align*}
\end{itemize}

For every $\varepsilon>0$, we shall now use the subordinators $(S^\mathbf{s})_{\mathbf{s} \in \mathcal{U}}$ to construct a coupling between $f$ and $\be$ such that $\P( \|f- \be \|>\varepsilon)<\varepsilon$. For every fixed $k \geq 1$, set  $\mathcal{R}_k := \{ a^\mathbf{s}, |\mathbf{s}| \leq k \} \cup \{ b^\mathbf{s}, |\mathbf{s}| \leq k \}$. Since $f$ and $\be$ both satisfy (ii) and (iii), we can couple them using the subordinators $(S^{\mathbf{s}})_{|\mathbf{s}| \leq k}$, so that a.s. 
$$\forall u \in \mathcal{R}_k, \qquad f(u) = \be(u).$$
Next, for every $\epsilon>0$, one can find $K_\epsilon \in \Z_+$ such that
$$\P \left( \sup \{ b-a:  a,b \in \mathcal{R}_{K_\epsilon}, (a,b) \cap \mathcal{R}_{K_\epsilon}=\varnothing \}>\epsilon \right) < \epsilon.$$
Furthermore, since $\be$ and $f$ are continuous on $[0,1]$, they are uniformly continuous. In particular, there exists $C_\eta>0$ such that $\P \left( \omega(\be) > C _{\eta} \right)<\eta$ and $\P \left( \omega(f) > C_\eta \right)<\eta$.
Now, on the event that $\{\sup \{ b-a : a,b \in \mathcal{R}_{K_\epsilon}, (a,b) \cap \mathcal{R}_{K_\epsilon}=\varnothing <\epsilon , \omega(\be) < C_\eta, \omega(f) < C_\eta\}$, we have clearly $\Vert f-\be \Vert < 2C_\eta \epsilon$. Thus $\P( \Vert f-\be \Vert > 2C_\eta \epsilon) \leq \varepsilon+2\eta$ (observe that the choice of $C_\eta$ is independent of $\epsilon$). This completes the proof.
\end{proof}

\begin{prop}
\label{prop:loie}
We have
\begin{align*}
(F(t), 0 \leq t \leq 1) \overset{(d)}{=} (\be_t, 0 \leq t \leq 1).
\end{align*}
\end{prop}

\begin{proof}
We need to prove that $F$ is continuous and satisfies (i), (ii), (iii) in Proposition \ref{prop:charaterization}. We immediately obtain (i) from Lemma \ref{lem:leaf} applied to the leaf $0$, continuity from Proposition \ref{prop:continuity} and (ii) from \cite[Theorem $1$ and Proposition $11$]{Ber00} combined with invariance by uniform rerooting of the Brownian CRT. Finally, (iii) comes from \cite[Corollary $2.3$]{BW17}.
\end{proof}

\begin{rk}\label{rem:decomposition}
The previous considerations entail that $ \mathcal{C}$ can be decomposed as follows, where we set  $P_{t}=P_{t}(F)$ and $y_t = P_{t-}-P_t$ to simplify notation.

\emph{Step 1.} The branch $\llbracket {\rho}, {\bar{\ell}({\rho})} \rrbracket$ isometric to a line segment with length $\int_0^\infty P_{s} \mathrm{d}s$.

\emph{Step 2.} For every $t>0$ such that $P_{t} < P_{t-}$ there is a branchpoint ${b}_t$ on $\llbracket {\rho}, {\bar{\ell}({\rho})} \rrbracket$ at distance $\int_0^t P_{s} \mathrm{d}s$ from the root, and on  ${b}_t$ is grafted the tree obtained by iteration with the function $F^{(t)}$ defined by $F^{(t)}=  y_{t}^{-1/2}(F({P_{t}+s y_{t}})-t(P_{t}+s y_{t})_{0 \leq s \leq 1}$, with distances renormalized by $y_t^{1/2}$. The leaf $\bar{\ell}(b_{t})$ is then the leaf associated with $b_{t}$ in the first step of this iteration. This construction provides us with an increasing sequence of trees (for the inclusion), and $\cC$ is the completion of the union of these trees.

\end{rk}

In order to see this, observe that each $t$ such that $P_t \neq P_{t-}$, we have $P_t=\mu_t(\varnothing)$. By monotonicity and density of such instants, this holds for all $t \geq 0$. In particular, for each branchpoint $b \in \llbracket \rho, \bar{\ell}(\rho) \rrbracket$ corresponding to a cutpoint appearing at time $t$, we have
\begin{align*}
d(\rho, b_t) = \int_0^t \mu_s(\varnothing) ds = \int_0^t P_s \mathrm{d}s.
\end{align*}
Letting $t \rightarrow \infty$, we have $d(\rho,\bar{\ell}(\rho))=\int_0^\infty P_s \mathrm{d}s$.
In order to show Step $2$, consider a branchpoint $b_t \in \llbracket \rho,\bar{\ell}(\rho) \rrbracket$ and a branchpoint $b'_u \in \llbracket b_t, \bar{\ell}(b_t) \rrbracket$ corresponding to a cutpoint $c$ appearing at time $u>t$. Again, it holds that
\begin{align*}
d(b_t,b'_u) = \int_t^u \mu_s(c) \mathrm{d}s = y_t^{1/2} \int_0^{y_t^{1/2}(u-t)} P_s(F^{(t)}) \mathrm{d}s,
\end{align*}
as $\mu_s(c)=\mu_t(c) \cdot P_{y_t^{1/2}(s-t)}(F^{(t)}) = y_t \cdot P_{y_t^{1/2}(s-t)}(F^{(t)})$ for all $s \geq t$ by the same argument.

\subsection{Proof of Theorem \ref{thm:couplage} (i)}
\label{ssec:proof}

Roughly speaking, to establish  Theorem \ref{thm:couplage} (i),  we shall consider the  function $F$ obtained by the Pac-Man algorithm from the Brownian CRT $\cT$ and the Poissonian rain $\mathcal{P}$. However, one has to be slightly careful since the Pac-Man algorithm has been defined using the cut-tree $\cC$, which is itself defined by using an additional source of randomness, namely the points $(U_{i})_{i \geq 1}$, so it is not clear whether the function $F$ defined this way is a measurable function of $(\cT,\cP)$.

To overcome this issue, we explain how the ``Bertoin function'' $F$ can be directly defined from $( \mathcal{T},\mathcal{P})$ in a measurable way.  Recall that for every $ t \geq 0$ and $ x\in \cT$, we denote by $\cT_{t}(x)$ the connected component of   $ \cT \backslash  \cP_{t}$ containing $x$ and $\mu_{t}(x)=\mu(\cT_{t}(x))$ its $\mu$-mass. We first set $F_{(\cT,\cP)}(0)=0$. Now, take $h \in (0,1]$. There are three cases:
\begin{itemize}[noitemsep,nolistsep]
\item[(i)] If there exists $t_1 \geq 0$ such that $\mu_{t_1}(\varnothing)=h$ , then we set $F_{(\cT,\cP)}(h)=t_1 h$.
\item[(ii)] If there exists $t_{1} \geq 0$ such that $\mu_{t_1-}(\varnothing)=h$, then we set $F_{(\cT,\cP)}(h)=t_1h$.
\item[(iii)] Otherwise, there exists $t_1 \geq 0$ such that $\mu_{t_1}(\varnothing) < h < \mu_{t_1-}(\varnothing)$.
\end{itemize}
In case (iii), notice that there exists a unique cutpoint $c_1 \in \cP_\infty$ that has appeared at time $t_1$. Furthermore, $c_1 \in \cT_{t_1-}(\varnothing)$. We now reason inductively. There are again three cases:
\begin{itemize}[noitemsep,nolistsep]
\item[(i)] If there exists $t_2 \geq t_1$ such that $\mu_{t_1}(\varnothing)+\mu_{t_2}(c_1)=h$, then we set $F_{(\cT,\cP)}(h)=t_1 \mu_{t_1}(\varnothing) + t_2 \mu_{t_2}(c_1)$.
\item[(ii)] If there exists $t_2 \geq t_1$ such that $\mu_{t_1}(\varnothing)+\mu_{t_2-}(c_1)=h$, then we set $F_{(\cT,\cP)}(h)=t_1 \mu_{t_1}(\varnothing)+t_2 \mu_{t_2-}(c_1)$.
\item[(iii)] Otherwise, there exists $t_2 \geq t_1$ such that $\mu_{t_2}(c_1) < h - \mu_{t_1}(\varnothing) < \mu_{t_2-}(c_1)$.
\end{itemize}
Thus, there are finally three cases, depending on $h$:

\begin{itemize}[noitemsep,nolistsep]
\item[(a)] either there exists a finite sequence $c_1, \ldots, c_k$ of cutpoints appeared at respective times $t_1, \ldots, t_k$ such that $h=\sum_{i=1}^k \mu_{t_i}(c_{i-1})$, $c_i \in \cT_{t_i-}(c_{i-1})$ for all $i \leq k$; in this case, $F_{(\cT,\cP)}(h)=\sum_{i=1}^k t_i \mu_{t_i}(c_{i-1})$;
\item[(b)] there exists a finite sequence $c_1, \ldots, c_k$ of cutpoints appeared at respective times $t_1, \ldots, t_k$ such that $h=\sum_{i=1}^{k-1} \mu_{t_i}(c_{i-1})+\mu_{t_k-}(c_{k-1})$, in which case $F_{(\cT,\cP)}(h)=\sum_{i=1}^{k-1} t_i \mu_{t_i}(c_{i-1})+t_k \mu_{t_k-}(c_{k-1})$.
\item[(c)] there exists an infinite sequence $c_1, \ldots$ of cutpoints appeared at respective times $t_1, \ldots, t_k$ such that $h=\sum_{i=1}^{\infty} \mu_{t_i}(c_{i-1})$, in which case $F_{(\cT,\cP)}(h)=\sum_{i=1}^\infty t_i \mu_{t_i}(c_{i-1})$.
\end{itemize}

\bigskip

We shall now prove that $F_{(\cT,\cP)}$ meets the requirements of Theorem \ref{thm:couplage} (i).

\begin{proof}[Proof of Theorem \ref{thm:couplage} (i)]
Consider a Brownian CRT $\cT$, the Poissonian rain $\mathcal{P}$  on $\cT$ and and let $U=(U_i)_{i \geq 1}$ be a sequence of i.i.d.~leaves of $\cT$ sampled according to the mass measure, independently of $\cP$. Denote by $F_{(\cT, {U}),\cP}$ the function defined by \eqref{eq:defF} using the Pac-Man algorithm. By the correspondence between cutpoints in $\cT$ and branchpoints in $\mathcal{C}$ and the associated masses (Lemma \ref{lem:bijectionsubtreescomponents}), almost surely $F_{(\cT, \mathcal{U},\cP )}=F_{(\cT,\cP)}$.

By Propositions \ref{prop:tailles} and \ref{prop:loie}, almost surely,  $F_{(\cT,\cP)}$ has the law of the Brownian excursion and almost surely, for every $t \geq 0$, the nonincreasing rearrangement of the masses of the connected components of $\mathcal{T} \backslash \cP_{t}$ is the same as the nonincreasing rearrangement of the lengths of the excursions of $(F_{(\cT,\cP)}(s)-ts)_{ 0 \leq s \leq 1}$ above its running infimum. This completes the proof.
\end{proof}

\section{Recovering the original tree together with its Poissonian rain}
\label{sec:recover}

Let $\cT$ be a Brownian CRT with mass measure $\mu$, $U \coloneqq (U_i)_{i \geq 1}$ a sequence of i.i.d. leaves of $\cT$ with common distribution $\mu$, and $\cP$ a Poissonian rain on $\Sk(\cT)$ independent of $U$. An important question in the literature (see \cite{ADG19, BHW22, BW17}) concerns the problem of reconstruction of the original tree: is it possible to reconstruct  $(\cT, U, \cP)$ being given the cut-tree? 

It turns out that there is a loss of information when one goes from a triple $(\cT, U, \cP)$ to the cut-tree $(\cutT, \N)$, where $\N$ denotes the subset $\{i: i \geq 1\}$ of points of $\mathcal{C}$ built in Sec.~\ref{ssec:cuttree}. More precisely, the following holds:
\begin{thm}{\cite[Theorem $3.2$ (c)]{BW17}}
Let $\cT$ be a Brownian CRT. Then there exists a (random) tree $\mathsf{shuff}(\cT)$ such that  in distribution:
\begin{align*}
(\cT, \mathsf{shuff}(\cT)) \overset{(d)}{=} (\mathrm{Cut}(\cT), \cT).
\end{align*}
\end{thm}

Here one recovers $\cT$ from $\mathrm{Cut}(\cT)$ only in distribution.
Later, Addario-Berry, Dieuleveut and Goldschmidt \cite{ADG19} have shown that, if one considers an enrichment of the cut-tree transform with information called \emph{routings}, it is possible to almost surely recover the initial tree (along with $U$ and $\cP$) from this enriched cut-tree (see \cite{BHW22} for an extension to Inhomogeneous Continuum Random Trees). In short, routing variables are a collection of uniform points in every subtree dangling on the branchpoints.

More precisely, recall that for every $b \in \mathcal{B}(\cC)$ the subtree $\cC_{b}$ of all descendents of $b$ is $\cC_{b}= \cC_{b}^{\ell(b)} \cup \cC_{b}^{\bar{\ell}(b)}$. Denote by $ \mathcal{S}$ the set of all subtrees of the form $ \cC_{b}^{\ell(b)}$ or $\cC_{b}^{\bar{\ell}(b)}$ for $b \in   \mathcal{B}(\cC)$.
Consider the set of so-called routing variables  $Z= (Z_{\mathcal{A}}, \mathcal{A} \in \mathcal{S})$ where $Z_{\cC_{b}^{\ell(b)}}$ is the image $\Lambda(\cC_{b}^{\ell(b)})$ of $b$ in $\cC_{b}^{\ell(b)}$ and $Z_{\cC_{b}^{\bar{\ell}(b)}}$ is the image $\Lambda(\cC_{b}^{ \bar{\ell}(b)})$ of $b$ in $\cC_{b}^{\bar{\ell}(b)}$.
 By \cite[Proposition 12]{ADG19},  for every $\mathcal{A} \in \mathcal{S}$ the random variable $Z_{\mathcal{A}}$ has law $\nu_{\mathcal{A}}$, and these random variables are conditionally independent given $(\cC, \nu,\N)$. We also have:

\begin{thm}{\cite[Proposition 12 \& Corollary $17$]{ADG19}}
\label{thm:ADG}
There exists a (deterministic) measurable map $\Phi$ such that, almost surely: $$\left(\cT, U, \cP\right) = \Phi\left((\cC, \N,Z)\right).$$
\end{thm}

The question answered by Theorem \ref{thm:couplage} (ii) is quite similar: being given the ``Bertoin'' Brownian excursion $\be$, is it possible to construct a map $\Psi$ such that $\Psi(\be)$ has the law of $(\mathcal{T},\mathcal{P})$ and $F_{\Psi(\be)}=\be$? The answer is positive, when adding an independent source of randomness.

The strategy of the proof is divided in two steps: first, in Sec.~\ref{ssec:B1} we show that having the  ``Bertoin'' excursion $F_{(\cT,\cP)}$ obtained from $(\cT,\cP)$ is equivalent to having the cut-tree $\cutT$, along with ``half'' of the routing variables. Second, in Sec.~\ref{ssec:B2} we add the additional information (the remaining ``half'' routings) that allows us to reconstruct a tree with a Poissonian rain.

\subsection{From Bertoin's excursion to the semi-enriched cut-tree}
\label{ssec:B1}

We first prove that having the ``Bertoin'' excursion $F_{(\cT,\cP)}$ obtained from $(\cT,\cP)$ is equivalent to having the cut-tree $\cutT$, along with a collection of points that we call half-routings, a notion which we shall now define. 

Let $T$ be a compact binary real tree with root $\rho$, and recall that $\cB(T)$ denotes the set of branchpoints of the tree $T$. We call half-routings on $T$ a collection of leaves ${H} \coloneqq \{ H_b, b \in \cB(T)  \}$ such that $H_b \in T_b$ for every $b \in \cB(T) $, where we recall that $T_{b}$ is set of all (weak) descendents of $b$ in $T$.
For every $b \in \cB(T)$, we define its associated record sequence $(b_i)_{i \geq 0} \in (\cB(T) \cup \{ \rho \})^{\Z_+}$ by induction as follows. Set $b_0=\rho$. Then for every $k \geq 0$, assuming that $(b_i)_{0 \leq i \leq k}$ has been defined, we set $b_{k+1} = \ell_{b_k} \wedge b$. Since $(b_k)_{k \geq 0}$ is increasing for the genealogical order, it converges in $T$.

We say that the collection $H$ of half-routings is \emph{consistent} if the following holds: 
\begin{align*}
\forall b \in \cB(T )\cup \{ \rho \}, \exists N_b \geq 0, b_{N_b}=b.
\end{align*}
In particular, when $H$ is consistent, then for all $b \in \cB(T)$ the sequence $(b_k)_{k \geq 0}$ is stationary after time $N_b$ and, for all $k \geq 1$ such that $b_k \neq b$, we have $\ell_{b_{k}} \in \overline{T}_{b_k}^{\ell_{b_{k-1}}}$. Indeed, by definition, $b_k = \ell_{b_{k-1}} \wedge b$, so that $b \in \overline{T}_{b_k}^{\ell_{b_{k-1}}}$. Thus, if we had $\ell_{b_{k}} \in T_{b_k}^{\ell_{b_{k-1}}}$, then it would follow that $b_{k+1}=\ell_{b_k} \wedge b = b_k$ and hence $b_j=b_k$ for $j \geq k$, so that $b=b_k$ since $H$ is consistent. Furthermore, every branchpoint has a finite record sequence. Finally, denote by $\mathbb{T}_{HR}$ the set of compact rooted binary real trees enriched with the a consistent collection of half-routings and a mass measure supported on its set of leaves.

Observe that the Pac-Man algorithm can be applied mutatis mutandis to any enriched tree $(T, H,M)$, where $H$ is a consistent collection of half-routings on $T$ and $M$ a probability measure on $T$. We denote by $\mathsf{X}(T, H,M)$ the function obtained from this algorithm.

We keep the notation $(\cT, U, \cP)$ for a Brownian CRT along with a sequence of i.i.d. leaves and a Poissonian rain on its skeleton, $\cutT$ its associated cut-tree, and set $\cH \coloneqq \left\{ \bar{\ell}(b), b \in \cB(\cutT)  \right\}$, with $\bar{\ell}(b)$ defined by \eqref{eq:ellellebar}. Observe that $\cH$ is a.s. a consistent collection of half-routings on $\cutT$. Indeed, it is clear by \eqref{eq:ellellebar} that $\bar{\ell}(b) \in \cC_b$ for all $b \in \cB(\cC) $. In addition, by Lemma \ref{lem:cvx}, the collection $\cH$ on $\cC$ is consistent. Denoting by $C\left( [0,1], [0,\infty] \right)$ the set of continuous maps from $[0,1]$ to $[0,\infty]$, we have the following result:

\begin{prop}
\label{prop:excequalcuttree}
There exists a (deterministic) map $\, \Xi: C\left( [0,1], [0,\infty] \right) \rightarrow \mathbb{T}_{HR}$ such that the following properties hold almost surely:
\begin{enumerate}[noitemsep,nolistsep]
 \item[(i)] We have $\Xi \circ \mathsf{X}(\cC,\cH,\nu)=(\cutT,\cH,\nu)$.
 \item[(ii)] Let $\be$ be a standard Brownian excursion. Then $\mathsf{X} \circ \Xi(\be)=\be$.
 \end{enumerate}
\end{prop}

\begin{proof}[Proof of Proposition \ref{prop:excequalcuttree}]
Let us start by defining the map $\Xi$. By construction, we have $\mathsf{X}(\cC,\cH,\nu)=F_{(\cT,\cP)}$. Since $F_{(\cT,\cP)}$ follows the law of a Brownian excursion by Theorem \ref{thm:couplage} (i), it is therefore enough to define $\Xi(\be)$. We use a stick-breaking construction in the spirit of \cite{AP00} to define $\Xi(\be)=(\tilde{\cC},\tilde{\cH},\tilde{\nu})$. We start with $ \{ \tilde{\rho}\}$, which will be the  root of the tree $\Xi(\be)$. 

\emph{Step 1.} Construct a branch $\llbracket \tilde{\rho}, {\tilde{\ell}(\tilde{\rho})} \rrbracket$ isometric to a line segment with length $\int_0^\infty P_{s} \mathrm{d}s$, where $ P_{s} \coloneqq \inf \{ u>0, \be_u = su \}$.

\emph{Step 2.} For every $t>0$ such that $P_{t} < P_{t-}$ we put a branchpoint $\tilde{b}_t$ on $\llbracket \tilde{\rho}, {\tilde{\ell}(\tilde{\rho})} \rrbracket$ at distance $\int_0^t P_s \mathrm{d}s$ from the root, and on  $\tilde{b}_t$ we graft the tree obtained by iteration with the function $\be^{(t)}$ defined by $\be^{(t)}=  (P_{t-} - P_{t})^{-1/2}(\be_{P_{t}+s(P_{t-}-P_{t})}-t(P_{t}+s(P_{t-}-P_{t}))_{0 \leq s \leq 1}$, which is a Brownian excursion by Proposition \ref{prop:charaterization}. In the first step of this iteration, a leaf, denoted by $\tilde{\ell}(\tilde{b}_t)$, is associated with $\tilde{b}_t$.

This construction provides us with an increasing sequence of trees (for the inclusion). We denote by  $\tilde{\cC}$  the completion of the union of these trees with a collection of half-routings $\tilde{\cH}=(\tilde{\ell}(\tilde{b}) : \tilde{b} \in \mathcal{B}(\tilde{\cC}) \cup \{ \tilde{\rho}\})$.

We check that $\tilde{\cC}$ is compact and establish (i) at the same time. Since the desired properties involve only the law of $\mathbb{e}$, without loss of generality, we may assume that $\mathbb{e}=\mathsf{X}(\cC,\cH,\nu)=F_{(\cT,\cP)}$. Observe that by Remark \ref{rem:decomposition}, the tree $\cC$ satisfies the same recursive construction as $\tilde{\cC}$. 
As a consequence, setting $\bar{\ell}(\rho)=0$ (which we recall to be the ``image'' in $\mathcal{C}$ of the root $\varnothing$ of $\cT$), the  trees $\cC^0 := \bigcup_{b \in \cB(\cC) \cup \{ \rho \}} \llbracket b, \bar{\ell}(b) \rrbracket$ and $\tilde{\cC}^0 := \bigcup_{\tilde{b} \in \cB(\tilde{\cC}) \cup \{ \tilde{\rho} \}} \llbracket \tilde{b}, \tilde{\ell}(\tilde{b}) \rrbracket$ are isometric, and so are their completions, which implies that  ${\cC}=\tilde{\cC}$.

Now let us explain how to endow $\tilde{\cC}$ with a mass measure. Roughly speaking, given $h \in [0,1]$, we explain how to define the final target point of the Pacman algorithm just from $\mathsf{X}(\cC,\cH,\nu)$. We construct a sequence of branchpoints associated with $h$ as follows. If there exists $t$ such that $P_{t} \neq P_{t-}$ and $h=P_{t}$ or $h=P_{t-}$, let $t_1(h)=\dagger$. Otherwise, let $t_1(h) := \inf \{t \in [0,1], P_{t} \neq P_{t-}, P_{t} < h \}$. From Step 2 above, with $t_1(h)$ is associated a branchpoint in $\tilde{\cC}$ denoted by $b_1(h)$.
Then, as for the stick breaking construction of $\cC$, we iterate this in the excursion $\be^{(t_1(h))}$. In the end, we obtain  an increasing sequence (for the genealogical order) of branchpoints associated with $h$. Also observe that this sequence stops at $\dagger$ only for countably many values of $h \in [0,1]$.
For every $h \in [0,1]$ for which this does not happen, we define $\ell(h)$ as the limit of the sequence $(b_i(h))_{i \geq 1}$.
Then given the description of the Pacman algorithm in the beginning of Sec.\ref{ssec:proof} we have $\ell(h)=\pi_{h}$, in the sense that $\ell(h)$ is the final target point of $h$ in $\cC$ by the Pacman algorithm.
We then define the mass measure $\tilde{\nu}$ on $\tilde{\cC}$ as the pushforward of the Lebesgue measure on $[0,1]$ by $\ell$, which by Lemma \ref{lem:measure} coincides with $\nu$.

We finally prove (ii). By step (i), we have $\mathsf{X}\left(\Xi(\be)\right)\overset{(d)}{=}\be$, so it suffices to check that they coincide a.s. on a dense subset of $[0,1]$. Now, for any $t \geq 0$ such that $P_{t-} \neq P_{t}$, by the discussion in the beginning of Sec.~\ref{ssec:proof} we have $\mathsf{X}\left(\Xi(\be)\right)(P_{t}) = \be_{P_{t}}$ and $\mathsf{X}\left(\Xi(\be)\right)(P_{t-}) = \be_{P_{t-}}$. Iterating this with the functions  $\be^{(t)}$, we obtain that a.s. $\mathsf{X}\left(\Xi(\be)\right)$ and $\be$ coincide on a dense set, and this completes the proof.
\end{proof}

\begin{rk}
\label{rk:rem2}
By Remark \ref{rk:rem1} the previous proof also shows that  if $ (\mathcal{C},\cH,\nu)=\Xi(\be)$ with $\cH= (\bar{\ell}(b), b \in  \cB(\cutT)  )$, then conditionally given $(\mathcal{C},\nu,\N)$,  $\bar{\ell}(b)$ has law $\nu_{\cC_{b}^{\bar{\ell}(b)}}$  for $b \in\cB(\cC)$, and these random variables are independent.
\end{rk}

\subsection{From the semi-enriched cut-tree to the initial tree}
\label{ssec:B2}

To establish Theorem \ref{thm:couplage} (ii) we apply the map $\Xi$ introduced in Proposition \ref{prop:excequalcuttree} to a Brownian excursion, which allows to reconstruct a cut-tree. We then want to apply  Theorem \ref{thm:ADG} to  reconstruct a Brownian CRT with its Poissonian rain, and to this end we need  a set of routing random variables. Additional randomness is required because $\Xi$ gives only ``half'' of the rootings.

\begin{proof}[Proof of Theorem \ref{thm:couplage} (ii)]
Consider a Brownian excursion $\mathbbm{e}$ and an independent sequence of i.i.d.~uniform random variables on $[0,1]$. Recall from  Proposition \ref{prop:excequalcuttree} the map $\Xi$. Set $ (\mathcal{C},\cH,\nu)=\Xi(\be)$ with $\cH= (\bar{\ell}(b), b \in  \cB(\cutT) )$. Recall also from the proof of Proposition \ref{prop:excequalcuttree} the map $\pi : [0,1] \rightarrow  \cC$.

For convenience, we split the i.i.d.~uniform random variables on $[0,1]$ into two independent collections of i.i.d~uniform random variables on $[0,1]$: the first one $(V_{i})_{i \geq 1}$ indexed by $\N$ and the second one $(W_b)_{b \in \cB(\cC)}$ indexed by the branchpoints of $\cC$ (this can be done in a deterministic measurable way).

We set $i=\pi(V_{i})\in \cC$ for every $i \geq 1$, so that $\N \cup \{\rho\}$ are i.i.d.~with law $\nu$, where $\rho$ is the root of $\mathcal{C}$.

We shall now define  routing variables  $Z= (Z_{\mathcal{A}}, \mathcal{A} \in \mathcal{S})$ such that for every $\mathcal{A} \in \mathcal{S}$ the random variable $Z_{\mathcal{A}}$ has law $\nu_{\mathcal{A}}$, and these random variables are conditionally independent given $(\cC, \nu)$, where we recall  that $ \mathcal{S}$ denotes the set of all subtrees of the form $ \cC_{b}^{\bar{\ell}(b)}$ or $\cC_b \backslash \cC_b^{\bar{\ell}(b)} \cup \{b\} $ for $b \in   \mathcal{B}(\cC)$. 

First, for every $b \in  \cB(\cutT) $ we set $Z_{\cC_{b}^{\bar{\ell}(b)}}=\bar{\ell}(b)$. Then,  by Remark \ref{rk:rem2} the random variables $(Z_{\cC_{b}^{\bar{\ell}(b)}}: b \in  \cB(\cutT) \cup \{\rho\} )$ have respective laws $\nu_{\cC_{b}^{\bar{\ell}(b)}}$, and these random variables are conditionally independent given $(\cC, \N)$.

Second take $b \in \cB(\cutT)$.  To define $Z_{\cC_b \backslash \cC_b^{\ell(b)} \cup \{b\} }$ we proceed as follows.  Assume for convenience that $b \in \llbracket \rho, 0 \rrbracket$. Keeping the notation introduced in the proof of Proposition \ref{prop:excequalcuttree}, with $b$ is associated a time $t \geq 0$ such that $P_{t} \neq P_{t-}$. Then set $Z_{\cC_b \backslash \cC_b^{\ell(b)} \cup \{b\}}=\pi_{ W_{b} P_{t}}$. Define it in the same way for every $b \in \cB(\cutT)$, in the spirit of the previous stick-breaking construction.

Finally define $Z= (Z_{\mathcal{A}}, \mathcal{A} \in \mathcal{S})$. By construction, for every $\mathcal{A} \in \mathcal{S}$ the random variable $Z_{\mathcal{A}}$ has law $\nu_{\mathcal{A}}$, and these random variables are conditionally independent given $(\cC, \N,\nu)$.

Applying to $(\cutT, \N, Z)$ the map $\Phi$ of \cite[Corollary $17$]{ADG19} provides a triple $\Phi(\cutT, \N, Z)$ having the law of $(\cT, U, \cP)$, such that almost surely, by Proposition \ref{prop:excequalcuttree}, $F_{\Phi(\cutT, \N, Z)} = \be$. This completes the proof.
\end{proof}

\bibliographystyle{abbrv}
\bibliography{BibliCutPrim}
\end{document}